\documentclass{amsart}
\usepackage{amssymb}
\usepackage[babel=true,kerning=true]{microtype}

\usepackage{amsthm,leqno}
\usepackage{hyperref, graphicx, ifpdf, amsfonts, amsmath}
\allowdisplaybreaks
\hypersetup{
   bookmarks=true,         
   unicode=false,          
   pdftoolbar=true,        
   pdfmenubar=true,        
   pdffitwindow=false,     
   pdfstartview={FitH},    
   pdftitle={My title},    
   pdfauthor={Author},     
   pdfsubject={Subject},   
   pdfcreator={Creator},   
   pdfproducer={Producer}, 
   pdfkeywords={keywords}, 
   pdfnewwindow=true,      
   colorlinks=true,       
   linkcolor=blue,          
   citecolor=blue,        
   filecolor=magenta,      
   urlcolor=MidnightBlue          
}

\usepackage{amsfonts,amsmath,amssymb,enumerate}
\usepackage{graphicx}
\usepackage{tikz}
\usetikzlibrary{decorations.markings}
\usetikzlibrary{plotmarks}
\usetikzlibrary{patterns}

\newcommand{\be}{\begin{equation*}}
\newcommand{\ee}{\end{equation*}}
\newcommand{\beq}{\begin{equation}}
\newcommand{\eeq}{\end{equation}}
\newcommand{\begincal}{\begin{eqnarray*}}
\newcommand{\fincal}{\end{eqnarray*}}

\newtheorem{thm}{Theorem}[section]
\newtheorem{lemma}[thm]{Lemma}
\newtheorem{cor}[thm]{Corollary}
\newtheorem{prop}[thm]{Proposition}

\newtheorem{rem}[thm]{Remark}

\newcommand{\eps}{{\varepsilon}}
\newcommand{\R}{{\mathbf R}}

\newcommand{\N}{{\mathbf N}}

\def\eps{\varepsilon}

\begin{document}

\title[Energy convexity of intrinsic bi-harmonic maps and applications]{Energy convexity of intrinsic bi-harmonic maps and applications I: spherical target}

\author{Paul Laurain}
\address{Institut de Math\'{e}matiques de JussieuUniversit\'{e} Paris Diderot, B\^{a}timent Sophie Germain\\ Case 7012, 75205 PARIS Cedex 13\\
France}
\email{paul.laurain@imj-prg.fr}

\author{Longzhi Lin}
\address{Mathematics Department\\University of California - Santa Cruz, 1156 High Street\\ Santa Cruz, CA 95064\\USA}
\email{lzlin@ucsc.edu}



\maketitle
\begin{abstract}
Every harmonic map is an intrinsic bi-harmonic map as an absolute minimizer of the intrinsic bi-energy functional, therefore intrinsic bi-harmonic map and its heat flow are more geometrically natural to study, but they are also considerably more difficult analytically than the extrinsic counterparts due to the lack of coercivity for the intrinsic bi-energy.  In this paper, we show an energy convexity and thus uniqueness for weakly intrinsic bi-harmonic maps from the unit $4$-ball $B_1 \subset \mathbf{R}^4$ into the sphere $\mathbf{S}^n$. This is a higher-order analogue of the energy convexity and uniqueness for weakly harmonic maps on unit $2$-disk in $\mathbf{R}^2$ proved by Colding and Minicozzi \cite{CM08} (see also Lamm and the second author \cite{LL13}). In particular, this yields a version of uniqueness of weakly harmonic maps on the unit $4$-ball which is new. As an application, we also show a version of energy convexity along the intrinsic bi-harmonic map heat flow into $\mathbf{S}^n$, which in particular yields the long-time existence of the intrinsic bi-harmonic map heat flow, a result that was until now only known assuming the non-positivity of the target manifolds by Lamm \cite{Lamm05}. Moreover, the energy convexity along the flow yields the uniform convergence of the flow which is not known before. One of the key ingredients in our proofs is a refined version of the $\epsilon$-regularity of the first author and Rivi\`{e}re \cite{LaR}.
\end{abstract}
\section{Introduction}
\subsection{Extrinsic and intrinsic bi-harmonic maps}
Let $(\mathcal{N}, g) \subset  \R^{n+1}$ be a closed (i.e. compact and without boundary) $C^3$ Riemannian submanifold with metric $g$ of $\R^{n+1}$. The \textit{extrinsic bi-energy} of a $W^{2,2}$ map $u = (u^1, ..., u^n) : B_1 \to \mathcal{N} \subset \R^{n+1}$ from the $4$-dimensional unit ball $B_1=B_1(0)\subset \R^4$ is defined by
\begin{equation}
E(u)\,=\,\frac{1}{4}\,\int_{B_1}\,|\Delta u|^2 dx\,,
\end{equation}
where $\Delta$ is the Laplacian operator on $\R^4$. A weakly \textit{extrinsic bi-harmonic map} $u$ from $B_1$ into $\mathcal{N}\hookrightarrow \R^{n+1}$ is a map in $W^{2,2}(B_1, \R^{n+1})$, which is a critical point of $E(u)$ and takes values almost everywhere in $\mathcal{N}$. 
Similarly, a weakly \textit{intrinsic bi-harmonic map} $u \in W^{2,2}(B_1, \mathcal{N})$ is a critical point of the \textit{intrinsic bi-energy}
\begin{equation}
\label{energy} I(u)\,=\,\frac{1}{4}\,\int_{B_1}\,|(\tau(u)|^2 dx \,=\, \,\frac{1}{4}\,\int_{B_1}\,|\Delta u|^2 - |A(u)(\nabla u, \nabla u)|^2dx\,,
\end{equation}
where $\tau(u) = (\Delta u)^T$ is the tangential component of $\Delta u$, which is commonly called the \textit{tension field}, and $A(u)$ is the second fundamental form of the embedding of $\mathcal{N}$ in $\mathbf{R}^{n+1}$.  Both the extrinsic and intrinsic bi-energies are scaling invariant in 4 dimensions. Now let $\Pi : \mathcal{N}_\delta \rightarrow \mathcal{N}$ be the nearest point projection map, which is well defined and $C^3$ for $\delta>0$ small enough. Here $\mathcal{N}_\delta=\{y\in\R^{n+1}\,\vert\, d\,(y, \mathcal{N})\leq \delta\}$. For $y\in \mathcal{N}$, let 
$$P(y)\equiv \nabla \Pi (y):\R^{n+1} \rightarrow T_y \mathcal{N}$$
be the orthogonal projection, and 
$$P^\bot(y)\equiv \text{Id}-\nabla\Pi (y):\R^{n+1} \rightarrow (T_y \mathcal{N})^\bot\,.$$
In the following, we will write $P$ (resp. $P^\bot$) instead of  $P(y)$ (resp. $P^\bot(y)$) and we will identify these linear transformations with their matrix representations in $\mathcal{M}_n$. We also note that these projections are in $W^{2,2}(B_1,\mathcal{M}_n)$ as soon as $u$ is in  $W^{2,2}(B_1,\mathcal{N})$ . Finally, note that the second fundamental form $A(\,\cdot\,)(\, \cdot\, , \,\cdot\,)$ of $\mathcal{N}\subset\R^{n+1}$ is defined by
$$A(y)(Y,Z)=D_Y P^\bot (y) (Z), \quad \forall \,y \in \mathcal{N} \quad \text{and} \quad Y , Z \in T_y \mathcal{N}.$$
We know that $u\in W^{2,2}(B_1,\mathcal{N})$ is an \textit{extrinsic bi-harmonic map} if and only if  (see e.g. \cite{Wa1})
$$ \Delta^2 u\,  \,\bot\,\, T_u \mathcal{N} \hbox{ almost everywhere},$$
which can be rewritten as follows:
\beq
\label{ebe}
\begin{split}
\Delta^2 u = &-\Delta (\nabla P^\bot  \nabla u)- \text{div}\, (\nabla P^\bot  \Delta u) + 2 \nabla P^\bot \nabla(\nabla P^\bot  \nabla u) \\
&+ 2 \nabla P^\bot \nabla P^\bot \Delta u -(\nabla P  P^\bot -P^\bot \nabla P) \nabla \Delta u.
\end{split}
\eeq
For \textit{intrinsic bi-harmonic map}, we need to add some tangent terms, see \cite{Wa2} and \cite{LaR} for details, and the equation is
\beq
\label{ibe}
\begin{split}
\Delta^2 u = &-\Delta (\nabla P^\bot  \nabla u)-\text{div}\, (\nabla P^\bot  \Delta u) + 2 \nabla P^\bot \nabla(\nabla P^\bot  \nabla u) + 2 \nabla P^\bot \nabla P^\bot \Delta u\\
&-(\nabla P  P^\bot -P^\bot \nabla P) \nabla \Delta u + P\left(\langle  \nabla P^\bot  \nabla u , \nabla_u(\nabla_u P^\bot)\nabla u  \nabla u \rangle \right) +  \langle \nabla P^\bot  \nabla u,  \nabla_u P^\bot \nabla u \nabla P\rangle  \\
&- \text{div} \,\langle \nabla P^\bot  \nabla u,  \nabla_u P^\bot  \nabla u P \rangle +\langle \nabla P^\bot  \nabla u,  \nabla P^\bot  \nabla P \rangle - \text{div} \,\langle \nabla P^\bot  \nabla u,  \nabla P^\bot  P \rangle\,.
\end{split}
\eeq
Here $\nabla_u P^\bot = \nabla_{y} P^\bot (y)|_{y=u}$.  Alternatively, when $\mathcal{N} = \mathbf{S}^n$,  the intrinsic bi-harmonic map $u$ satisfies the equation (see e.g. Lamm-Rivi\`{e}re \cite{LR})
\begin{equation}\label{ibe02}
\Delta^2 u \,=\, \Delta (V\cdot \nabla u) + \text{div }(w \nabla u) + W\cdot \nabla u\,,
\end{equation}
where
\beq
\label{ibe03}
\left\{
\begin{aligned}
V^{ij} &=  u^i\nabla u^j - u^j\nabla u^i\\
w^{ij}&= \text{div}\left(V^{ij}\right) \\
W^{ij}&= \nabla w^{ij}+ 2 \left[ \Delta u^i \nabla u^j - \Delta u^j \nabla u^i + |\nabla u|^2(u^i\nabla u^j - u^j\nabla u^i) \right]\,.
\end{aligned}
\right.
\eeq
It is worth noting that intrinsically the intrinsic bi-harmonic map equation \eqref{ibe} can be written as
$$
\Delta_g^2 u - R^{\mathcal{N}}(\nabla_g u, \tau(u))\nabla_g u = 0\,,
$$
where $R^{\mathcal{N}}$ is the curvature tensor of $\mathcal{N}$. This was first derived by Jiang in \cite{Jiang86}.

\vskip 2mm
Harmonic maps are critical points of the Dirichlet energy $D(u)$ given by
$$
D(u) = \frac{1}{2}\int_{B_1}|\nabla u|^2 dx\,,
$$
and they are clearly critical points (absolute minimizers) of the intrinsic bi-energy $I(u)$ since they satisfy $\tau(u) = 0$. In this sense, the intrinsic bi-harmonic map is a more geometrically natural generalization of the harmonic map, although the extrinsic bi-harmonic map is more natural from the analytical point of view. It shall be noted that the extrinsic bi-energy $E(u)$, in contrast to $I(u)$, depends on the embedding of the target manifold $\mathcal{N}$ in $\mathbf{R}^{n+1}$.  Despite the fact that it is more geometrically natural to study the intrinsic bi-harmonic map, it is less studied and considerably more difficult analytically than the extrinsic bi-harmonic map due to the lack of \textit{coercivity} for the intrinsic bi-energy $I(u)$ (because the two terms $\Delta u$ and $A(u)(\nabla u, \nabla u)$ in the intrinsic bi-energy $I(u)$ can dissolve but cannot dominate each other in an analytic way). The existence and non-existence of nontrivial (or \textit{proper}, i.e., non-harmonic) \textit{intrinsic} bi-harmonic maps can be found in Jiang \cite{Jiang86}, Mou \cite{Mou00}, Caddeo-Montaldo-Oniciuc \cite{CMO01, CMO02}, Oniciuc \cite{Oni02} and Baird-Kamissoko \cite{BK03}. For the regularity of weakly intrinsic bi-harmonic maps on 4 dimensional Euclidean domains, see e.g. Ku \cite{Ku08} for spherical targets and Moser \cite{Mos06}, Wang \cite{Wa1} for general targets. The regularity of weakly \textit{extrinsic} bi-harmonic maps has been studied by Chang-Wang-Yang \cite{CWY99}, Strzelecki \cite{Str03} and Wang \cite{Wa1, Wa2}, see also Lamm-Rivi\`{e}re \cite{LR}. 

\begin{rem}
There is a notion of bi-harmonic sub-manifolds (intrinsic bi-harmonic isometric immersion $\mathcal{M} \to \mathcal{N}$), which generalizes the notion of minimal sub-manifolds (harmonic isometric immersion $\mathcal{M} \to \mathcal{N}$) since every harmonic map is an intrinsic bi-harmonic map. It is conjectured that any bi-harmonic sub-manifold in sphere has parallel mean curvature vector (commonly referred to as the \text{generalized Chen's conjecture}), see Ou \cite{Ou16} for a survey.
\end{rem}
\begin{rem}
The divergence terms $- \text{div} \,\langle \nabla P^\bot  \nabla u,  \nabla_u P^\bot  \nabla u P \rangle$ and $- \text{div} \,\langle \nabla P^\bot  \nabla u,  \nabla P^\bot  P \rangle$ were missing in the intrinsic bi-harmonic map equation \eqref{ibe} in Wang \cite{Wa2} and all other related literatures. The analytical estimates needed for regularity still hold for these missing terms and thus the results in \cite{Wa2} etc. are still valid, but for our energy convexity and uniqueness results in this paper it is important to consider all terms.
\end{rem}

\begin{rem}
In local coordinates, the last several terms in \eqref{ibe} read as
\beq\label{localex}
\begin{split}
P\left(\langle  \nabla P^\bot  \nabla u , \nabla_u(\nabla_u P^\bot)\nabla u  \nabla u \rangle \right) \,&=\,\sum_{\alpha, \beta,\gamma, i, j, k, m} P_{lk} D_\alpha (P^\bot)_{ij} D_{\alpha} u^j  \nabla_{u^k} \nabla_{u^\beta}(P^\bot)_{im}  D_\gamma u^\beta D_\gamma u^m \,;\\
 \langle \nabla P^\bot  \nabla u,  \nabla_u P^\bot \nabla u \nabla P\rangle \,&=\, \sum_{\alpha, \beta, \gamma, i, j, k, m} D_\alpha (P^\bot)_{ij} D_{\alpha} u^j  \nabla_{u^\beta}(P^\bot)_{im} D_\gamma u^m D_\gamma P_{\beta l}\,;\\
 \langle \nabla P^\bot  \nabla u,  \nabla P^\bot  \nabla P \rangle\,&=\, \sum_{\alpha, \beta, i, j, k} D_\alpha (P^\bot)_{ij} D_{\alpha} u^j  D_\beta (P^\bot)_{ik}    D_\beta P_{kl}\,;\\
\text{div} \,\langle \nabla P^\bot  \nabla u,  \nabla_u P^\bot  \nabla u P \rangle \,&=\,  \sum_{\alpha, \beta, \gamma, i, j, k} D_\beta \left [D_\alpha (P^\bot)_{ij} D_{\alpha} u^j   \nabla_{u^\gamma}(P^\bot)_{ik}   D_\beta u^k P_{\gamma l}\right]\,;\\
\text{div} \,\langle \nabla P^\bot  \nabla u,  \nabla P^\bot  P \rangle \,&=\,  \sum_{\alpha, \beta, i, j, k} D_\beta \left [D_\alpha (P^\bot)_{ij} D_{\alpha} u^j  D_\beta (P^\bot)_{ik}    P_{kl}\right]\,;
\end{split}
\eeq
\end{rem}

\subsection{Extrinsic and intrinsic bi-harmonic map heat flows}
The negative $L^2$-gradient flow of the extrinsic bi-energy $E(u)$ is called a (weakly) extrinsic bi-harmonic map heat flow. It is given by $u\in W^{1,2}([0,T), L^2(B_1, \mathcal{N} )) \cap L^2([0, T); W^{2,2}(B_1, \mathcal{N}))$ which satisfies
\beq
\label{bfe}
\left\{
\begin{aligned}
 \frac{\partial u}{\partial t} + \Delta^2 u = &-\Delta (\nabla P^\bot  \nabla u)- \text{div}\, (\nabla P^\bot  \Delta u) + 2 \nabla P^\bot \nabla(\nabla P^\bot  \nabla u) \\
&+ 2 \nabla P^\bot \nabla P^\bot \Delta u -(\nabla P  P^\bot -P^\bot \nabla P) \nabla \Delta u
 \quad \text{ on } B_1 \times [0, T)\,; \\
u= & \,u_0 \quad  \text{ on } B_1\times \{0\}\,;\\
u= & \,\chi(x)\,, \quad \partial_{\nu} u = \xi(x) \quad  \text{ on } \partial B_1\times [0, T)\,.
\end{aligned}
\right.
\eeq
where $u_0 \in W^{2,2}(B_1, \mathcal{N})$.  There are several results of existence for extrinsic bi-harmonic map heat flow, see for instance Lamm \cite{Lamm04} for small initial data and Gastel \cite{Gas06} and Wang \cite{Wa4} for solutions with many finitely singular time and any initial data, see also \cite{HHW14} and \cite{HLLW16}. Moreover, the solutions to \eqref{bfe} satisfy the following energy inequality
\beq
\label{ei} 
2\int_{0}^T \int_{B_1} \left\vert \frac{\partial u}{\partial t} \right\vert^2 \, dx\, dt + \int_{B_1 \times \{T\}} \vert \Delta u\vert^2 \, dx \leq   \int_{B_1} \vert\Delta u_0\vert^2 \, dx\,.
\eeq

Similarly,  a weakly intrinsic bi-harmonic map heat flow is the negative $L^2$-gradient flow of the intrinsic bi-energy $I(u)$ given by $u\in W^{1,2}([0,T), L^2(B_1, \mathcal{N} )) \cap L^2([0, T); W^{2,2}(B_1, \mathcal{N}))$ which satisfies
\beq
\label{bfi}
\left\{
\begin{aligned}
\frac{\partial u}{\partial t} + \Delta^2 u = &-\Delta (\nabla P^\bot  \nabla u)-\text{div}\, (\nabla P^\bot  \Delta u) + 2 \nabla P^\bot \nabla(\nabla P^\bot  \nabla u) + 2 \nabla P^\bot \nabla P^\bot \Delta u\\
&-(\nabla P  P^\bot -P^\bot \nabla P) \nabla \Delta u + P\left(\langle  \nabla P^\bot  \nabla u , \nabla_u(\nabla_u P^\bot)\nabla u  \nabla u \rangle \right)   \\
&+  \langle \nabla P^\bot  \nabla u,  \nabla_u P^\bot \nabla u \nabla P\rangle - \text{div} \,\langle \nabla P^\bot  \nabla u,  \nabla_u P^\bot  \nabla u P \rangle  \\
&+\langle \nabla P^\bot  \nabla u,  \nabla P^\bot  \nabla P \rangle- \text{div} \,\langle \nabla P^\bot  \nabla u,  \nabla P^\bot  P \rangle \quad \text{ on } B_1 \times [0, T)\,; \\
u= & \,u_0 \quad  \text{ on } B_1\times \{0\}\,;\\
u= & \,\chi(x)\,, \quad \partial_{\nu} u = \xi(x) \quad  \text{ on } \partial B_1\times [0, T)\,.
\end{aligned}
\right.
\eeq
where $u_0 \in W^{2,2}(B_1,\mathcal{N})$. The solutions to \eqref{bfi} satisfy the following energy inequality
\beq
\label{eii} 
2\int_{0}^T \int_{B_1} \left\vert \frac{\partial u}{\partial t} \right\vert^2 \, dxdt + \int_{B_1  \times \{T\}} \vert \tau(u) \vert^2 \, dx \leq   \int_{B_1} \vert\tau(u_0)\vert^2 \, dx \,.
\eeq
Intrinsically, the intrinsic bi-harmonic map heat flow equation \eqref{bfi} can be written as
\begin{equation}
\label{ibmhf1}
\frac{\partial u}{\partial t} +\Delta^2 u =  R^{\mathcal{N}}(\nabla u, \tau(u))\nabla u \,.
\end{equation}

Contrary to the extrinsic bi-harmonic map heat flow, there is no bi-energy monotonicity for the intrinsic bi-harmonic map heat flow (cf. \eqref{ei} and \eqref{eii} for $T>0$), hence in general only short time existence of a smooth solution is known for the intrinsic bi-harmonic map heat flow, see e.g. Lamm \cite{Lamm01}, Mantegazza-Martinazzi \cite{MM12}. In \cite{Lamm05}, under the assumption that the source manifold is at most four-dimensional (and closed) and the sectional curvature of the target manifold $\mathcal{N}$ is \textit{non-positive}, Lamm proved that the intrinsic bi-harmonic map heat flow with smooth initial data has a global smooth solution, which \textit{sub-converges} to a harmonic map as the time tends to infinity. As we mentioned above, the main difficulty of the study of the intrinsic bi-harmonic map and its heat flow is the lack of \textit{coercivity} for the intrinsic bi-energy $I(u)$. The non-positivity of the sectional curvature on the target manifold $\mathcal{N}$ in \cite{Lamm05} removes the lack of coercivity to some extent. For general target manifolds, see e.g. Moser \cite{Mos05} for a discussion of the blow-up behavior of the intrinsic bi-harmonic map heat flow on 4 dimensional domains as it approaches the first singular time, although \textit{no} such example is known. One of the main results in this paper (Theorem \ref{THM2} and Corollary \ref{COR2}) is to give the first result regarding the \textit{long time existence and uniform convergence} of the intrinsic bi-harmonic map heat flow without assuming the non-positivity of the target manifolds, which is open for many years since Lamm \cite{Lamm05}.

\subsection{Main results}
The main results of this paper are the following. We denote $\nu$ to be the outward unit normal vector on $\partial B_1$.

\begin{thm}\label{THM1}
There exists a constant $\varepsilon_0>0$ such that if $u, v\in W^{2,2}(B_1, \mathbf{S}^n)$ with $u|_{\partial B_1}=v|_{\partial B_1},\partial_\nu u |_{\partial B_1}=\partial_\nu v|_{\partial B_1},$ 
$$\int_{B_1} \vert \Delta u\vert^2 \, dx \,=\, \int_{B_1} \vert \tau(u)\vert^2 + \vert \nabla u\vert^4 \, dx \leq \varepsilon_0 \quad \text{and}\quad \int_{B_1} \vert \nabla v\vert^4 \, dx \leq \varepsilon_0\,,$$
and $u$ is a weakly intrinsic bi-harmonic map, then we have the energy convexity
\begin{equation}\label{Convexity1}
\frac{1}{8}\int_{B_1} |\Delta v-\Delta u|^2\,\leq\, \frac{1}{2}\int_{B_1} |\tau(v)-\tau(u)|^2 \,\leq\, \int_{B_1} |\tau(v)|^2-\int_{B_1}|\tau(u)|^2\,.
\end{equation}
\end{thm}

\begin{rem}\label{auto}
One of the benefits when the target manifold is $\mathbf{S}^n$ is that we automatically have $\int_{B_1} \vert \nabla u\vert^4 \, dx \,\leq\, \int_{B_1} \vert \Delta u\vert^2 \, dx $.
\end{rem}

An immediate corollary of Theorem \ref{THM1} is the uniqueness of weakly intrinsic bi-harmonic maps into spheres with small bi-energy on $B_1$.
\begin{cor}\label{COR1}
There exists a constant $\varepsilon_0>0$ such that for all weakly intrinsic bi-harmonic maps (in particular, weakly harmonic maps) $u, v \in W^{2,2}(B_1, \mathbf{S}^n)$ with energies 
$$\int_{B_1} \vert \Delta u\vert^2 \, dx \leq \varepsilon_0\,,\quad \int_{B_1}  \vert \Delta v\vert^2 \, dx \leq \varepsilon_0\,,$$
and $u|_{\partial B_1}=v|_{\partial B_1},\partial_\nu u |_{\partial B_1}=\partial_\nu v|_{\partial B_1}$, we have $u \equiv v$ in $B_1$.
\end{cor}

\begin{rem} \label{remextrinsic} The corresponding energy convexity and uniqueness for weakly harmonic maps from the unit $2$-disk into general closed target manifolds were proved by Colding and Minicozzi in \cite{CM08}, see also Lamm and the second author \cite{LL13}. The energy convexity for weakly extrinsic bi-harmonic maps from the unit $4$-ball $B_1$ into $\mathbf{S}^n$ was shown in Hineman-Huang-Wang \cite[Theorem 3.2]{HHW14}, under the stronger assumption $$\int_{B_1} \vert \nabla^2 u\vert^2\, dx \leq\varepsilon_0\,,$$ 
we provide a simpler self-contained proof of this result for spherical targets under the assumption 
$$\int_{B_1} \vert \Delta u\vert^2\, dx \leq\varepsilon_0\,,$$ 
see Theorem \ref{extrinsicbiharmonic}. Note also that the smallness condition $\int_{B_1} \vert \Delta u\vert^2 \, dx = \int_{B_1} \vert \tau(u)\vert^2 + \vert \nabla u\vert^4\, dx \leq\varepsilon_0$ is necessary: a harmonic map is an intrinsic bi-harmonic map as an absolute minimizer of the intrinsic bi-energy functional, but its Dirichlet energy $\int_{B_1} \vert \nabla u\vert^2 \,dx$ could be large and the uniqueness is not possible, c.f. Brezis-Coron \cite{BC83}.
\end{rem}

Thanks to the robustness of our method, which essentially relies on an improved $\eps$-regularity result that follows from the technics developed by Lamm-Rivi\`ere \cite{LR} and the first author with Rivi\`ere \cite{LaR}, we are able to prove a similar energy convexity along the intrinsic bi-harmonic heat flow into $\mathbf{S}^n$ treated as a perturbation of the stationary equation. One of the key observations of this paper is that, such energy convexity along the flow yields the \textit{coercivity} for the intrinsic bi-energy $I(u)$ and thus proving the long time existence and uniform convergence of the intrinsic bi-harmonic map heat flow into $\mathbf{S}^n$ with small initial bi-energy. The corresponding results for the extrinsic bi-harmonic map heat flow with small initial bi-energy into general target manifolds were claimed in \cite[Corollary 1.8]{HHW14} (which also follows directly from our method for spherical targets, c.f. Theorem \ref{extrinsicbiharmonic}), but it seems they need a smallness condition on the initial Hessian energy rather than the bi-energy in order to prove such results, see \cite[equations (1.14) and (5.58)]{HHW14}. We will address the intrinsic bi-harmonic map and its heat flow into general target manifolds in a forthcoming paper.

\begin{thm}\label{THM2}
There exists a constant $\varepsilon_0>0$ such that if $u_0 \in C^\infty (\overline{B_1},\mathbf{S}^n)$ with \begin{equation}\label{initialcondition}
\int_{B_1} |\Delta u_0|^2 \, dx \leq \varepsilon_0\,,
\end{equation}
then the initial-boundary value problem for the intrinsic bi-harmonic map heat flow \eqref{bfi} has a smooth solution $u \in C^{\infty}(\overline{B_1} \times [0,\infty), \mathbf{S}^n)$ for all time. Moreover, there exists $T_1>0$ such that along the flow there holds an energy convexity
\begin{equation}\label{MainConv}
\frac{1}{16}\int_{B_1} |\Delta u (\cdot, t_1)-\Delta u(\cdot, t_2)|^2\,\leq \, \int_{B_1} |\tau(u)(\cdot, t_1)|^2-\int_{B_1}|\tau(u)(\cdot, t_2)|^2
\end{equation}
for all $t_2 > t_1\geq T_1$.
\end{thm}

Immediate application of Theorem \ref{THM2} is the following corollary.

\begin{cor}\label{COR2}
There exists a constant $\varepsilon_0>0$ such that if $u_0 \in C^\infty (\overline{B_1},\mathbf{S}^n)$ with $\int_{B_1} \vert \Delta u_0\vert^2\, dx \leq \varepsilon_0$, then the initial-boundary value problem for the intrinsic bi-harmonic map heat flow \eqref{bfi} has a smooth solution $u \in C^{\infty}(\overline{B_1}\times [0,\infty), \mathbf{S}^n)$ such that
\begin{equation} u(\cdot,t) \to u_{\infty} \text{ uniformly as } t \to +\infty \text{ strongly in } W^{2,2}(B_1, \mathbf{R}^{n+1})\,,
\end{equation}
where $u_{\infty}$ is the unique smooth intrinsic bi-harmonic map with $u_\infty|_{\partial B_1} = \chi$ and $\partial_\nu u_\infty|_{\partial B_1} = \xi$.
\end{cor}

\begin{rem}
 The corresponding uniform convergence result for harmonic map heat flow on the $2$-disk was proved by the second author in \cite{Lin13}, see also Wang \cite{Wang12}. \end{rem}
 
 The paper is organized as follows. In Section \ref{s1} we prove Theorem \ref{THM1}, the energy convexity and uniqueness of weakly intrinsic bi-harmonic maps with small bi-energy into spheres. We also give a self-contained proof of the energy convexity and uniqueness of weakly extrinsic bi-harmonic maps with small bi-energy into $\mathbf{S}^n$ (Theorem \ref{extrinsicbiharmonic}). We defer the proofs of the $\eps$-regularity theorem for approximate bi-harmonic maps (Theorem \ref{ereg}) and a technical lemma (Lemma \ref{lf}) to Appendices \ref{a1} and \ref{APPB} respectively. In Section \ref{section2} we prove Theorem \ref{THM2}, the long time existence and uniform convergence of the intrinsic bi-harmonic map heat flow with small initial bi-energy into spheres.
 

\section{Energy convexity and uniqueness of intrinsic bi-harmonic maps}
\label{s1}
This section devotes to the proof of Theorem \ref{THM1}, which contains most of the main ingredients for the the proof of Theorem \ref{THM2} (we will just treat the $u_t = \frac{\partial u}{\partial t}$ as a perturbation term thanks to the $\eps$-regularity Theorem \ref{ereg}). We first need the following second order Hardy inequality, see e.g. Edmunds-R\'akosn\'\i k \cite{ER99} or Hineman-Huang-Wang \cite[Lemma 3.1]{HHW14} (there is a typo in the statement of \cite[Lemma 3.1]{HHW14}: the $f$ in the statement should also satisfy the zero Neumann boundary condition).

\begin{thm}[Hardy inequality]\label{hardyineq}
There exists a constant $C>0$ such that if $w \in W^{2,2}(B_1, \mathbf{R}^{n+1})$ with $w|_{\partial B_1}=0, \partial_\nu w |_{\partial B_1}= 0,$ then we have
\begin{equation}\label{hardyest}
\int_{B_1} \vert w\vert^2  (1-|x|)^{-4} \; dx  \,\leq \,C \int_{B_1} \vert \Delta w\vert^2\;dx \,.
\end{equation}
\end{thm}

Next we give a proof of the energy convexity and uniqueness for weakly \textit{extrinsic} bi-harmonic maps into $\mathbf{S}^n$ that were already shown by Hineman, Huang and Wang in \cite{HHW14} under the stronger assumption $\int_{B_1} \vert \nabla^2 u\vert^2\, dx \leq\varepsilon_0$, see Remark \ref{remextrinsic}. 

\begin{thm}\label{extrinsicbiharmonic}
There exists a constant $\varepsilon_0>0$ such that if $u, v\in W^{2,2}(B_1, \mathbf{S}^n)$ with $u|_{\partial B_1}=v|_{\partial B_1},\partial_\nu u |_{\partial B_1}=\partial_\nu v|_{\partial B_1},$ 
$$\int_{B_1} \vert \Delta u\vert^2 \, dx \leq \varepsilon_0,,$$
and $u$ is a weakly extrinsic bi-harmonic map, then we have the energy convexity
\begin{equation}
\frac{1}{2}\int_{B_1} |\Delta v- \Delta u|^2\,\leq\,\int_{B_1} |\Delta v\vert^2 - \vert\Delta u|^2\,.
\end{equation}
Therefore, if additionally $v$ is also a weakly extrinsic bi-harmonic map with $\int_{B_1} \vert \Delta v\vert^2 \, dx \leq \varepsilon_0,,$ then $u \equiv v$ in $B_1$.
\end{thm}

\begin{proof}
For extrinsic bi-harmonic maps we have:
\beq
\begin{split}
P^\perp (\Delta^2 u) \,= \,& \Delta^2 u\,,
\end{split}
\eeq
and thus
\be
\begin{split}
& \int_{B_1} |\Delta v|^2 \; dx - \int_{B_1} |\Delta u|^2 \; dx - \int_{B_1} |\Delta ( v - u)|^2 \\
=& 2  \int_{B_1} \langle \Delta^2 u, v-u\rangle\; dx \,=\, 2 \int_{B_1} \langle P^\perp (\Delta^2 u), v-u\rangle\; dx \\
\geq & -C  \int_{B_1} \vert (v-u)^\bot \vert \cdot\vert P^\perp \left(\Delta^2 u\right)\vert \; dx\\
\geq & - C\varepsilon_0 \int_{B_1} \vert v - u\vert^2 (1-|x|)^{-4}\; dx\\
\geq & - C\varepsilon_0 \int_{B_1}|\Delta (v - u)|^2\; dx\,,
\end{split}
\ee
where we used the fact that $\vert (v-u)^\bot \vert \leq C|v-u|^2$ for some $C >0$, see e.g. \cite[Lemma A.1]{CM08}\,, the $\eps$-regularity Theorem \ref{ereg}  (with $f\equiv 0$) and the Hardy inequality Theorem \ref{hardyineq}. Choosing $\varepsilon_0$ sufficiently small yields the desired energy convexity.
\end{proof}

Now let us focus on the energy convexity and uniqueness for weakly \textit{intrinsic} bi-harmonic maps into $\mathbf{S}^n$.

\begin{lemma}\label{conv2}
There exists a constant $\varepsilon_0>0$ such that if $u, v\in W^{2,2}(B_1, \mathbf{S}^n)$ with $u|_{\partial B_1}=v|_{\partial B_1},\partial_\nu u |_{\partial B_1}=\partial_\nu v|_{\partial B_1},$ 
\begin{equation}\label{smallcond}
\int_{B_1} \vert \Delta u\vert^2 \, dx \,=\, \int_{B_1} \vert \tau(u)\vert^2 + \vert \nabla u\vert^4 \, dx \leq\varepsilon_0 \quad \text{and}\quad\int_{B_1} \vert \nabla v\vert^4 \, dx \leq \varepsilon_0\,,
\end{equation}
and $u$ is a weakly intrinsic bi-harmonic map, then we have 
\begin{equation}
\int_{B_1} |\Delta (v-u)|^2 \, dx\,\leq\, 4\int_{B_1} |\tau(v)-\tau(u)|^2 \, dx\,.
\end{equation}
\end{lemma}
\begin{proof}
Note that
\beq
\begin{split}
\int_{B_1} |\Delta (v - u)|^2  \, dx = &\,\int_{B_1} \left|\tau(v)-\tau(u) - v|\nabla v|^2 + u|\nabla u|^2\right|^2 \, dx\\
\leq & \,2\int_{B_1} |\tau(v)-\tau(u)|^2 \, dx + 2 \int_{B_1} \left|v|\nabla v|^2 - u|\nabla u|^2\right|^2 \, dx\\
= &\,2\int_{B_1} |\tau(v)-\tau(u)|^2 \, dx + 2 \int_{B_1} \left|v(|\nabla v|^2 - |\nabla u|^2) + (v-u) |\nabla u|^2\right|^2 \, dx\\
\leq & \,2\int_{B_1} |\tau(v)-\tau(u)|^2 \, dx + 4 \int_{B_1} \left| |\nabla v|^2 - |\nabla u|^2\right|^2  \, dx+ 4\int_{B_1}|v-u|^2 | \nabla u|^4 \, dx\\
\leq & \,2\int_{B_1} |\tau(v)-\tau(u)|^2 \, dx + 4 \left(\int_{B_1} |\nabla (v+u)|^4  \, dx \right)^{\frac{1}{2}}\left(\int_{B_1} |\nabla (v-u)|^4  \, dx \right)^{\frac{1}{2}}\\
&+ 4\varepsilon_0\int_{B_1}|\Delta (v-u)|^2 \, dx\\
\leq &\, 2\int_{B_1} |\tau(v)-\tau(u)|^2 \, dx + C\sqrt{\varepsilon_0} \int_{B_1}|\Delta (v-u)|^2 \, dx\,,
\end{split}
\eeq
where we have used \eqref{smallcond} and
\begin{equation}
\label{control1}
\left(   \int_{B_1} \vert \nabla (v-u) \vert^4  \; dx \right)^\frac{1}{4} \leq C \left(   \int_{B_1} \vert \Delta (v-u) \vert^2  \; dx \right)^\frac{1}{2}\,.
\end{equation} Choosing $\varepsilon_0$ sufficiently small so that $C\sqrt{\varepsilon_0} \leq \frac{1}{2}$ yields the desired estimate.
\end{proof}

\begin{lemma}\label{bienergycon}
There exists a constant $\varepsilon_0>0$ such that if $u, v\in W^{2,2}(B_1, \mathbf{S}^n)$ with $u|_{\partial B_1}=v|_{\partial B_1},\partial_\nu u |_{\partial B_1}=\partial_\nu v|_{\partial B_1},$ 
\begin{equation}
\int_{B_1} \vert \Delta u\vert^2 \, dx \,=\, \int_{B_1} \vert \tau(u)\vert^2 + \vert \nabla u\vert^4 \, dx \leq\varepsilon_0\,,
\end{equation}
and $u$ is a weakly intrinsic bi-harmonic map, then we have 
$$
\int_{B_1} |\Delta v|^2 \; dx - \int_{B_1} |\Delta u|^2 \; dx - \int_{B_1} |\Delta (v-u)|^2 \; dx \geq - C\varepsilon_0 \int_{B_1} |\Delta (v-u)|^2 \; dx + 4 \int_{B_1} |\nabla u|^2\nabla u \nabla (v-u) \; dx\,.
$$
\end{lemma}
\begin{proof}
For intrinsic bi-harmonic maps we have:
\beq
\begin{split}
P^\perp (\Delta^2 u) \,= \,& \Delta^2 u - P\left(\langle  \nabla P^\bot  \nabla u , \nabla_u(\nabla_u P^\bot)\nabla u  \nabla u \rangle \right) -  \langle \nabla P^\bot  \nabla u,  \nabla_u P^\bot \nabla u \nabla P\rangle \\
&+ \text{div} \,\langle \nabla P^\bot  \nabla u,  \nabla_u P^\bot  \nabla u P \rangle
-  \langle \nabla P^\bot  \nabla u,  \nabla P^\bot  \nabla P \rangle + \text{div} \langle \nabla P^\bot  \nabla u,  \nabla P^\bot  P \rangle\,,
\end{split}
\eeq
and therefore
\begin{align}
& \,\int_{B_1} |\Delta v|^2 \; dx - \int_{B_1} |\Delta u|^2 \; dx - \int_{B_1} |\Delta (v - u)|^2\; dx \notag\\
=&\,2 \int_{B_1} \langle \Delta^2 u, v-u\rangle\; dx \label{temp4} \\
= &\, 2 \int_{B_1} \langle P^\perp (\Delta^2 u) + P\left(\langle  \nabla P^\bot  \nabla u , \nabla_u(\nabla_u P^\bot)\nabla u  \nabla u \rangle \right) +  \langle \nabla P^\bot  \nabla u,  \nabla_u P^\bot \nabla u \nabla P\rangle, v-u\rangle\; dx \notag\\
&\,-2 \int_{B_1}  \left\langle \text{div} \,\langle \nabla P^\bot  \nabla u,  \nabla_u P^\bot  \nabla u P \rangle -\langle \nabla P^\bot  \nabla u,  \nabla P^\bot  \nabla P \rangle + \text{div} \langle \nabla P^\bot  \nabla u,  \nabla P^\bot  P \rangle, v-u \right\rangle\; dx\notag\\
= &\,2 \int_{B_1} \langle P^\perp (\Delta^2 u), v-u\rangle dx + 2 \int_{B_1}  \left\langle \langle  \nabla P^\bot  \nabla u , \nabla_u(\nabla_u P^\bot)\nabla u  \nabla u \rangle , P(v-u)\right\rangle\; dx\notag\\
&\,+2 \int_{B_1}  \left\langle  \langle \nabla P^\bot  \nabla u,  \nabla_u P^\bot \nabla u \nabla P\rangle , v-u\right\rangle\; dx + 2 \int_{B_1}  \left \langle \langle \nabla P^\bot  \nabla u,  \nabla_u P^\bot  \nabla u P \rangle, \nabla(v-u) \right\rangle \;dx \notag\\
&\,+ 2 \int_{B_1} \langle\langle \nabla P^\bot  \nabla u, \nabla P^\bot  \nabla P\rangle , v-u\rangle \;dx + 2 \int_{B_1}   \langle \langle \nabla P^\bot  \nabla u,  \nabla P^\bot  P \rangle, \nabla(v-u)\rangle \;dx \notag\\
= &\, \text{I} + \text{II} + \text{III} + \text{IV} + \text{V} + \text{VI}\,.\notag
\end{align}
Next we will estimate these six terms one by one.

\subsection{Term I} Note that term I can be handled using the $\eps$-regularity Theorem \ref{ereg} and Hardy inequality Theorem \ref{hardyineq}, so that
\beq \label{termI}
\text{I}\geq -C  \int_{B_1} \vert (v-u)^\bot \vert \cdot\vert P^\perp \left(\Delta^2 u\right)\vert \; dx\geq -C  \int_{B_1} \vert v-u \vert^2 (1-|x|)^{-4} \; dx \,\geq\, - C\varepsilon_0 \int_{B_1}|\Delta (v - u)|^2\; dx \,,
\eeq
where we have used the fact that $\vert (v-u)^\bot \vert \leq C|v-u|^2$ for some $C>0$. This estimate for term I is in fact valid for general target manifolds.

\subsection{Term II}  As in \eqref{localex}, the integrant in term II reads as
$$ \sum_{\alpha, \beta,\gamma, i, j, k, m,l} P_{lk} D_\alpha (P^\bot)_{ij} D_{\alpha} u^j  \nabla_{u^k} \nabla_{u^\beta}(P^\bot)_{im}  D_\gamma u^\beta D_\gamma u^m (v-u)^l\,.$$
Now using the fact that the target manifold is $\mathbf{S}^n$, we know that $u$ is the unit normal vector at the point $u\in  \mathbf{S}^n$ and $P^\bot({\mathbf{v}}) = \langle {\mathbf{v}}, u\rangle u$ for any vector ${\mathbf{v}} \in T_u(\mathbf{S}^n)$ so that
$$
P^\bot_{ij} = u^i u^j \quad\text{and}\quad P^\bot (\Delta u) = - \nabla P^\bot \nabla u = - u |\nabla u|^2\,.
$$
Therefore, in this case the integrant in term II becomes
\beq
\label{termII}
\begin{split}
\sum_{\beta,\gamma, i, k, m,l} P_{lk}   u_i|\nabla u|^2 (\delta_{i\beta}\delta_{mk} + \delta_{ik}\delta_{m\beta}) D_\gamma u^\beta D_\gamma u^m (v-u)^l \,=\,0\,.
\end{split}
\eeq

\subsection{Term III}  Note that $$P_{ij} = \delta_{ij} - u_iu_j\,,$$ and therefore the integrant in term III becomes
\beq
\label{termIII}
\begin{split}
 & \sum_{\alpha, \beta, \gamma, i, j, k, m,l} D_\alpha (P^\bot)_{ij} D_{\alpha} u^j  \nabla_{u^\beta}(P^\bot)_{im} D_\gamma u^m D_\gamma P_{\beta l}(v-u)^l \,= \,- |\nabla u|^2 \nabla u^\beta\cdot(\nabla u^\beta u^l + u^\beta\nabla u^l)(v-u)^l\\
=&\,- |\nabla u|^4 u\cdot (v-u) = - |\nabla u|^4 u \cdot (v-u)^\bot\, \geq\, - C|\nabla u|^4 |v-u|^2\,.
\end{split}
\eeq
Therefore, term II can be handled similarly as term I using Theorem \ref{ereg} and Theorem \ref{hardyineq}.

\subsection{Term IV}  Note that, as in \eqref{localex},  the integrant in term IV is (using again $P^\bot_{ij} = u_i u_j$)
$$
\sum_{\alpha, \beta, \gamma, i, j, k,l}  D_\alpha (P^\bot)_{ij} D_{\alpha} u^j   \nabla_{u^\gamma}(P^\bot)_{ik}   D_\beta u^k P_{\gamma l} D_\beta(v-u)^l \,=\, |\nabla u|^2 \nabla u \cdot \nabla (v-u)\,.
$$

\subsection{Term V}  Using again $P^\bot_{ij} = u_i u_j, P + P^\bot = \text{Id}$ and $u\cdot \nabla u =0$ on $\mathbf{S}^n$ we have
\be
\begin{split}
&2 \int_{B_1} \langle\langle \nabla P^\bot  \nabla u, \nabla P^\bot  \nabla P\rangle , v-u\rangle \;dx = 2  \int_{B_1} \nabla P^\bot_{ij} \nabla u^j\nabla P^\bot_{ik}\nabla P_{ks}(v-u)^s \;dx
\\
&= - 2  \int_{B_1} |\nabla u|^4 u \cdot (v-u) \;dx \geq -2\int_{B_1} |\nabla u|^4 |(v-u)^\bot| \;dx \geq -C\int_{B_1} |v-u|^2|\nabla u|^4  \;dx\,.
\end{split}
\ee
Therefore it can be handled similarly as term I using Theorem \ref{ereg} and Theorem \ref{hardyineq}, and we get
\beq \label{termV}
\begin{split}
\text{V} \,\geq\, - C\varepsilon_0 \int_{B_1}|\Delta (v - u)|^2\; dx \,.
\end{split}
\eeq

\subsection{Term VI} 
For term VI, we have (using again $P^\bot_{ij} = u^i u^j$)
\beq \label{termVI}
\begin{split}
\text{VI} \,=\,2\int_{B_1}   \langle \langle \nabla P^\bot  \nabla u,  \nabla P^\bot  P \rangle, \nabla(v-u)\rangle \; dx \,&=\, 2\int_{B_1}   |\nabla u|^2\nabla u \cdot \nabla (v-u) \; dx\,.
\end{split}
\eeq
Combining \eqref{psi}, \eqref{termI}, \eqref{termII}, \eqref{termV} and \eqref{termVI}  we have
\beq
\begin{split}
&\int_{B_1} \vert \Delta v\vert^2 \; dx -  \int_{B_1} \vert \Delta u\vert^2 \; dx - \int_{B_1} \vert \Delta(v-u) \vert^2 \; dx\\
&\geq - C\varepsilon_0 \int_{B_1} \vert \Delta(v-u) \vert^2 \; dx+  4\int_{B_1}   |\nabla u|^2\nabla u \cdot \nabla (v-u) \; dx\,.
\end{split}
\eeq
This completes the proof of the lemma.
\end{proof}

\begin{proof}(of Theorem \ref{THM1})
It suffices to prove that 
$$\psi \geq -\frac{1}{2} \int_{B_1} |\tau(v)-\tau(u)|^2 \, dx $$
where
\beq \label{psi}
\begin{split}
\psi&=  \int_{B_1} \vert \tau(v)\vert^2 \; dx -  \int_{B_1} \vert\tau(u)\vert^2 \; dx - \int_{B_1} \vert \tau(v) - \tau(u)\vert^2 \; dx\\
=&\, \int_{B_1} |\Delta v|^2 \; dx - \int_{B_1} |\Delta u|^2 \; dx - \int_{B_1} |\Delta (v-u)|^2 \; dx - \int_{B_1} |\nabla v|^4 - |\nabla u|^4\; dx\\
& + \int_{B_1} \left| v|\nabla v|^2 + u|\nabla u|^2\right|^2 \; dx + 2 \int_{B_1} |\nabla u|^2 u \Delta v + |\nabla v|^2 v\Delta u \; dx\\
\geq &\,- C\varepsilon_0 \int_{B_1} |\Delta (v-u)|^2 \; dx + 4 \int_{B_1} |\nabla u|^2\nabla u \nabla (v-u) \; dx- \int_{B_1} |\nabla v|^4 - |\nabla u|^4\; dx\\
& + \int_{B_1} \left| v|\nabla v|^2 + u|\nabla u|^2\right|^2 \; dx + 2 \int_{B_1} |\nabla u|^2 u \Delta v + |\nabla v|^2 v\Delta u \; dx\,,
\end{split}
\eeq
where we have used Lemma \ref{bienergycon}. Therefore,  using $u|_{\partial B_1}=v|_{\partial B_1}$ and $\partial_\nu u |_{\partial B_1}=\partial_\nu v|_{\partial B_1}$, we have

\beq 
\begin{split}
\psi\geq  &\,- C\varepsilon_0 \int_{B_1} |\Delta (v-u)|^2 \; dx - 2 \int_{B_1} \nabla |\nabla u|^2 u \nabla v \; dx - 2 \int_{B_1}  |\nabla u|^2 u \Delta v \; dx \\
& - 2 \int_{B_1} \nabla |\nabla u|^2 v \nabla u \; dx - 2 \int_{B_1} |\nabla u|^2 v \Delta u \; dx - 4 \int_{B_1}|\nabla u|^4\; dx - \int_{B_1} |\nabla v|^4 - |\nabla u|^4\; dx\\
& + \int_{B_1} \left| v|\nabla v|^2 + u|\nabla u|^2\right|^2 \; dx + 2 \int_{B_1} |\nabla u|^2 u \Delta v + |\nabla v|^2 v\Delta u \; dx\\
=  &\,- C\varepsilon_0 \int_{B_1} |\Delta (v-u)|^2 \; dx + \int_{B_1} \nabla |\nabla u|^2 \nabla |v-u|^2 \; dx + 2 \int_{B_1}  v\Delta u (|\nabla v|^2 - |\nabla u|^2)  \; dx \\
& + 2\int_{B_1} |\nabla u|^2 uv (|\nabla v|^2 - |\nabla u|^2) \; dx - \int_{B_1} |v-u|^2|\nabla u|^4 \; dx\,,
\end{split}
\eeq
where we have used $|u|^2 = |v|^2 = 1$ so that $1-uv = \frac{1}{2}|v-u|^2$ and $u\nabla v + v\nabla u = -\frac{1}{2}\nabla |v-u|^2$. Thus, we have
\beq 
\begin{split}
\psi& \geq  \,- C\varepsilon_0 \int_{B_1} |\Delta (v-u)|^2 \; dx + 2 \int_{B_1} (\Delta u + u|\nabla u|^2) v(|\nabla v|^2 - |\nabla u|^2) \; dx \\
& = \,- C\varepsilon_0 \int_{B_1} |\Delta (v-u)|^2 \; dx + 2 \int_{B_1}\tau(u) \cdot (v-u) \langle\nabla (v+u), \nabla (v-u)\rangle \; dx \\
&\geq\, - C\varepsilon_0 \int_{B_1} |\Delta (v-u)|^2 \; dx - 2 \left(\int_{B_1}|\tau(u)|^2|v-u|^2\; dx\right)^{\frac{1}{2}} \Vert \nabla (v+u) \Vert_4\Vert \nabla (v-u) \Vert_4 \\
&\geq \,- C\sqrt{\varepsilon_0} \int_{B_1} |\Delta (v-u)|^2 \; dx \,,
\end{split}
\eeq
where we used \eqref{control1}, Theorem \ref{ereg}, Theorem \ref{hardyineq} and the condition $\int_{B_1} \vert \nabla u\vert^4 + \vert \nabla v\vert^4 \, dx \leq 2\varepsilon_0$. Choosing $\varepsilon_0$ sufficiently small so that $C\sqrt{\varepsilon_0} \leq \frac{1}{2}$ and combining with Lemma \ref{conv2} give the desired energy convexity \eqref{Convexity1}.
\end{proof}


\section{Long time existence and uniform convergence} \label{section2}
In this section we prove Theorem \ref{THM2}. First of all we need to control the $L^2$-norm of  $u_t=\frac{\partial u}{\partial t}$, which is the goal of the next proposition.
\begin{prop}
\label{p1} 
Under the same assumption as Theorem \ref{THM2}, there exist $0<T_1<T_2$ such that for all $t_1,t_2 \in [T_1, T_2] \subset [0,T_0], t_1 < t_2$ we have
\begin{equation}\label{decreasingofu_t}
\int_{B_1} \left\vert u_t (.,t_2) \right\vert^2 \; dx \leq  \frac{1}{t_2 -t_1 } \int_{t_1}^{t_2} \int_{B_1} \left\vert u_t \right\vert^2 \; dx \; dt \,.
\end{equation}
Here $T_0>0$ is the short time of existence of the flow.
\end{prop}

\begin{proof}
Differentiating the flow equation \eqref{ibmhf1} or the first equation in \eqref{bfi} with respect to $t$, multiplying by $u_t$ and integrating over $B_1 \times [t_1, t_2]$, we have
\begin{align}
&\frac{1}{2}\int_{t_1}^{t_2}\int_{B_1} \partial_t |u_t|^2  \; dx \; dt+ \int_{t_1}^{t_2}\int_{B_1} |\Delta u_t|^2  \; dx \; dt\notag\\
\leq &\, C \int_{t_1}^{t_2}\int_{B_1}  |u_t|^2(|\nabla u|^4 + |\nabla u|^2|\nabla^2 u| + |\nabla u||\nabla \Delta u| + |\nabla^2 u|^2)  \; dx \; dt\notag\\
& + C \int_{t_1}^{t_2}\int_{B_1}  |u_t||\Delta u_t| (|\nabla u|^2 + |\nabla^2 u|) + |\nabla u_t||\Delta u_t||\nabla u|  \; dx \; dt \label{temp1}\\
\leq &\, C \int_{t_1}^{t_2}\int_{B_1}|u_t|^2(|\nabla u|^4 + |\nabla u|^2|\nabla^2 u| + |\nabla u||\nabla \Delta u| + |\nabla^2 u|^2 )  \; dx \; dt+  \frac{1}{2}\int_{t_1}^{t_2}\int_{B_1} |\Delta u_t|^2   \; dx \; dt\,,\label{temp119}
\end{align}
where we have used the Young's inequality. We also used the fact that
\begin{align*}
\int_{B_1}  |\nabla u_t|^2|\nabla u|^2  \; dx &= -\int_{B_1}  u_t \Delta u_t |\nabla u|^2  + u_t \nabla u_t \nabla |\nabla u|^2 \; dx\\
& \leq \frac{1}{\alpha}\int_{B_1} |\Delta u_t|^2   \; dx + C(\alpha) \int_{B_1} |u_t|^2 (|\nabla u|^4+|\nabla^2 u|^2) \; dx + \frac{1}{2}\int_{B_1}  |\nabla u_t|^2|\nabla u|^2  \; dx 
\end{align*}
and (with the last term on the right-hand side being absorbed into the left-hand side) then we chose $\alpha>0$ appropriately to accommodate the constant $C>0$ in \eqref{temp1} in order to get the factor $1/2$ in the last term of \eqref{temp119}. Now by the general short time existence result of Lamm \cite{Lamm01} or Mantegazza-Martinazzi \cite{MM12}, there exists $T_0>0$ depending only on $\|u_0\|_{C^\infty(\overline{B_1})}$ such that the flow exists smoothly up to $T_0$. Thanks to \eqref{eii} and the small initial bi-energy \eqref{initialcondition} we may assume without loss of generality that at $T_1 \in [0, T_0]$ we have
$$
\int_{B_1} |u_t|^2 (\cdot, T_1) \,dx \leq \varepsilon_0 \quad\text{and}\quad \int_{B_1} |\Delta u|^2 (\cdot, T_1) \,dx \leq 2\varepsilon_0\,.
$$ 
To see this, since $u(x, t)$ and all its derivatives are uniformly bounded by a constant depending only on $\|u_0\|_{C^\infty(\overline{B_1})}$ for all $t\in [0,T_0]$,  there exists $T_2>0$ depending only on $\|u_0\|_{C^\infty(\overline{B_1})}$ such that we have 
$$\int_{B_1} |\Delta u|^2 (\cdot, t) \,dx \leq 2\varepsilon_0 \quad\text{for all } t\in [0,T_2] \subset [0,T_0]\,.$$
Therefore if $\int_{B_1} |\tau(u_0)|^2 \,dx \leq \int_{B_1} |\Delta u_0|^2 \,dx < 2T_2 \varepsilon_0$ then by \eqref{eii} there exists $T_1\in [0,T_2)$ such that $\int_{B_1} |u_t|^2 (\cdot, T_1) \,dx \leq \varepsilon_0$. Note that now the choice of $\varepsilon_0$ depends on $\|u_0\|_{C^\infty(\overline{B_1})}$. Now by continuity of $\int_{B_1}|u_t(\cdot,t)|^2 dx$ in $t$, there exists $\delta>0$ depending only on $\|u_0\|_{C^\infty(\overline{B_1})}$ such that for any $t_0\in [T_1, T_1+\delta] \subset [T_1, T_2]$ we have
\begin{equation}\label{temp3}
\int_{B_1} |u_t|^2 (\cdot, t_0) \,dx \leq 2\varepsilon_0 \quad \text{and}\quad \Vert u_t (\cdot, t_0)\Vert_{W^{2,2}(B_1) \hookrightarrow L^p(B_1)} \leq C \varepsilon_0^{1/6}\,.
\end{equation}
Here $C>0$ depends only on $\|u_0\|_{C^\infty(\overline{B_1})}$ and $p<\infty$. For any $t_0\in [t_1, t_2]$, we can apply the $\epsilon$-regularity Theorem \ref{ereg} for approximate bi-harmonic maps and use the same arguments as in the proof of the Hardy inequality Theorem \ref{hardyineq} to get estimate for
$$\int_{B_1}|u_t|^2(|\nabla u|^4 + |\nabla u|^2|\nabla^2 u| + |\nabla u||\nabla \Delta u| + |\nabla^2 u|^2 ) (\cdot, t_0)\,dx\,.$$
Indeed, similar to the proof of \eqref{hardyest} (replacing $v-u$ by $u_t(t_0)$), for any $t_0\in [t_1, t_2]$ we have
\begin{equation}\label{temp2}
\int_{B_1}|u_t|^2(|\nabla u|^4 + |\nabla u|^2|\nabla^2 u| + |\nabla u||\nabla \Delta u| + |\nabla^2 u|^2 ) (\cdot, t_0)\,dx \leq C\varepsilon_0^{1/6}\int_{B_1}|\Delta u_t(\cdot, t_0)|^2\,dx\,.
\end{equation}
Here we need to rewrite the flow equation as $\Delta^2 u = -u_t +Q(u)$ and treat $f = -u_t$ as the perturbation term in the $\epsilon$-regularity Theorem \ref{ereg},  where
\begin{align}\label{Quu}
Q(u) = &\,P^\perp (\Delta^2 u) + P\left(\langle  \nabla P^\bot  \nabla u , \nabla_u(\nabla_u P^\bot)\nabla u  \nabla u \rangle \right) + \langle \nabla P^\bot  \nabla u,  \nabla_u P^\bot \nabla u \nabla P\rangle \notag\\
&- \text{div} \,\langle \nabla P^\bot  \nabla u,  \nabla_u P^\bot  \nabla u P \rangle
+  \langle \nabla P^\bot  \nabla u,  \nabla P^\bot  \nabla P \rangle - \text{div} \langle \nabla P^\bot  \nabla u,  \nabla P^\bot  P \rangle\,,
\end{align}
see also \eqref{ibe02} and \eqref{ibe03}. Using \eqref{temp3} (say $\Vert u_t \Vert_{L^5 \hookrightarrow L^{4,1}} \leq C\varepsilon_0^{1/6}$) and thanks to \eqref{Fereg} of Theorem \ref{ereg} we get
$$\vert Q(u)\vert \leq C\varepsilon_0^{1/6} (1-\vert x\vert)^{-4} $$ 
and then we can apply Theorem \ref{hardyineq} to get the desired estimate \eqref{temp2}. Inserting \eqref{temp2} back into \eqref{temp1} (for any $t_0\in [t_1, t_2]$) we see that the right-hand side of \eqref{temp2} can be absorbed into the the left-hand side if we choose $\varepsilon_0$ sufficiently small. 
This implies that we have 
\begin{equation}
\int_{B_1} |u_t(\cdot, t_2)|^2 \,\leq\,  \int_{B_1} |u_t(\cdot, t_1)|^2
\end{equation}
for $T_1 \leq t_1< t_2\leq T_1+\delta$. This shows, instead of \eqref{temp3}, for any $t_0 \in [T_1, T_1+\delta]$ we have
\begin{equation}
\int_{B_1} |u_t|^2 (\cdot, t_0) \,dx \leq \int_{B_1} |u_t|^2 (\cdot, T_1) \,dx\leq \varepsilon_0\,.
\end{equation}
We can then continue and iterate this process from $T_1+\delta$ to $T_1 + 2\delta$ and so on, noting that $\Vert u(\cdot, T_1+k \delta) \Vert_{C^\infty(\overline{B_1})}, k=1,2,3,...$ are uniformly bounded for all $k$ (as long as $T_1 + k\delta \leq T_2$) by $C\varepsilon_0^{1/6}$ thanks to Theorem \ref{ereg} and bootstrapping (so that $\delta>0$ has a definite size and we can iterate this process until the time hits $T_2$). We see that $\int_{B_1} |u_t(\cdot, t)|^2\,dx$ is indeed non-increasing along the flow after $T_1$, which yields \eqref{decreasingofu_t} for all $t_1, t_2 \in [T_1,T_2]$ with $t_1 < t_2$. In the above calculations, we should treat $u_t$ as a difference quotient: $u_t(\cdot,t) = \lim_{h\to 0^+}(u(\cdot, t+h)-u(\cdot,t))/h$ which has zero Dirichlet and Neumann boundary conditions on $\partial B_1$ for all $t\in [0,T_1]$;  moreover, we have denoted $\Delta u_t(\cdot,t) = \lim_{h\to 0^+}(\Delta (u(\cdot, t+h)-u(\cdot,t)))/h$ and all the calculations are valid for any fixed $h>0$ and then we take $h\to 0^{+}$ to conclude \eqref{temp2}. This completes the proof of the proposition.
\end{proof}

\begin{proof}(of Theorem \ref{THM2})
We start with the case that $t_1, t_2 \in [T_1,T_2]$ where $T_1$ and $T_2$ are from Proposition \ref{p1}. Similar to the proof of Lemma \ref{conv2}, using the $\epsilon$-regularity Theorem \ref{ereg} we have
\begin{equation}
\int_{B_1} |\Delta u_1-\Delta u_2 |^2 \, dx\,\leq\, 4\int_{B_1} |\tau(u_1)-\tau(u_2)|^2 \, dx\,,
\end{equation}
if we choose $\varepsilon_0$ sufficiently small. Here we denote $u_i=u(\cdot, t_i)$.  To get the energy convexity \eqref{MainConv}, it suffices to prove
$$\psi \geq -\left( \int_{B_1} \vert \tau(u_1)\vert^2 \; dx -  \int_{B_1} \vert\tau(u_2)\vert^2 \; dx \right) -\frac{1}{2} \int_{B_1} |\tau(u_1)-\tau(u_2)|^2 \, dx $$
where
$$
\psi=  \int_{B_1} \vert \tau(u_1)\vert^2 \; dx -  \int_{B_1} \vert\tau(u_2)\vert^2 \; dx - \int_{B_1} \vert \tau(u_1) - \tau(u_2)\vert^2 \; dx\,.
$$

The rest of the proof can be taken verbatim from the the proof of Theorem \ref{THM1} with one modification in the proof of Lemma \ref{bienergycon}. Namely, in the proof of Lemma \ref{bienergycon}, we will need to replace $\Delta^2 u$ by $-u_t +Q(u)$ instead of just $Q(u)$ in equation \eqref{temp4}, where $Q(u)$ is the same as \eqref{Quu}. Therefore, the only extra term in the estimate \eqref{temp4} we need to take into account is
\beq
\begin{split}
2 \int_{B_1}  \langle -(u_t)_2, u_1 -u_2\rangle \; dx &\geq - 2 \left(  \int_{B_1}  \vert (u_t)_2  \vert^2 \; dx \right)^\frac{1}{2}  \left(  \int_{B_1}  \vert u_1-u_2  \vert^2 \; dx \right)^\frac{1}{2} \\
& \geq - 2 \left(  \int_{B_1}  \vert (u_t)_2  \vert^2 \; dx \right)^\frac{1}{2}  \sqrt{t_2-t_1}\left(\int_{t_1}^{t_2}  \int_{B_1}  \vert u_t  \vert^2 \; dx \; dt \right)^\frac{1}{2} \\
&  \geq - 2 \int_{t_1}^{t_2}  \int_{B_1}  \vert u_t  \vert^2 \; dx \; dt \\
& \geq  -\left( \int_{B_1} \vert \tau(u_1)\vert^2 \; dx -  \int_{B_1} \vert\tau(u_2)\vert^2 \; dx \right)\,,
\end{split}
\eeq
where $(u_t)_i=u_t(\cdot, t_i)$ and the last inequality is a consequence of  \eqref{eii}. This shows the energy convexity \eqref{MainConv} when $t_1, t_2 \in [T_1, T_2]$. Now thanks to the energy convexity \eqref{MainConv}, for any $t\in [T_1, T_2]$, the bi-energy $\int_{B_1}|\Delta u (\cdot, t)|$ is \textit{uniformly} bounded, which is small if we choose $\varepsilon_0$ sufficiently small. Therefore the $\eps$-regularity Theorem \ref{ereg} allows us to continue and iterate this process from $T_2$ to $2T_2-T_1$ and so on, and the flow exists smoothly for all time and we have the energy convexity \eqref{MainConv} along the flow. This completes the proof of Theorem \ref{THM2}. 
\end{proof}

\appendix
\section{$\eps$-regularity for approximate bi-harmonic maps}
\label{a1}
First, we recall the main result of Lamm-Rivi\`{e}re \cite{LR} that provides a divergence form to elliptic fourth order system of certain type (see e.g. \eqref{nf}) under small energy assumption. This will be one of the main tools in order to obtain the estimates needed for the energy convexity for intrinsic and extrinsic bi-harmonic maps into $\mathbf{S}^n$, see Theorem \ref{ereg}. The first three results in this appendix work for any general closed target manifold $\mathcal{N}$, viewed as a submanifold of $\R^{n+1}$ (of any co-dimension).
\begin{prop}
\label{pn}
The  equation \eqref{ebe} and  \eqref{ibe} can be rewritten in the form
\beq
\label{nf}
\Delta^2 u = \Delta (V \nabla u) + div( w \nabla u) + \nabla \omega \nabla u + F\nabla u ,
\eeq
where $V \in W^{1,2}(B_1, \mathcal{M}_{n+1} \otimes \Lambda^{1} \R^4)$, $w\in L^2(B_1, \mathcal{M}_{n+1})$, $\omega \in L^2(B_1, \mathrm{s}o_{n+1})$ and $F \in L^2\cdot W^{1,2}(B_1,  \mathcal{M}_{n+1} \otimes \Lambda^{1} \R^4)$ with
\beq
\label{cc}
\begin{split}
&\vert V\vert  \leq C\vert \nabla u\vert\,,\\
& \vert F \vert  \leq C \vert\nabla u \vert\left(\vert \nabla^2 u\vert + \vert \nabla u\vert^2 \right)\,,\\
&\vert w \vert + \vert \omega \vert  \leq C\left( \vert \nabla^2 u\vert + \vert \nabla u\vert^2 \right ) \,,
\end{split}
\eeq
almost everywhere, where $C>0$ is a constant which depends only on $\mathcal{N}$ .
\end{prop}

\begin{thm}[{\cite[Theorem 1.4]{LR}}]
\label{TLR0} There exists $\eps>0$ and $C>0$ depending only on $\mathcal{N}$, such that the following holds:  let  $V \in W^{1,2}(B_1, \mathcal{M}_{n+1} \otimes \Lambda^{1} \R^4)$, $w\in L^2(B_1, \mathcal{M}_{n+1})$, $\omega \in L^2(B_1,\mathrm{s}o_{n+1})$ and  $F \in L^2\cdot W^{1,2}(B_1,  \mathcal{M}_{n+1} \otimes \Lambda^{1} \R^4)$ be such that
$$\Vert V \Vert_{W^{1,2}} +\Vert w \Vert_2 + \Vert \omega \Vert_2 + \Vert F \Vert_{L^2 \cdot W^{1,2}} < \eps,$$
then there exist $A\in L^\infty \cap W^{2,2}(B_1, \mathcal{G}l_{n+1})$ and $B\in  W^{1,\frac{4}{3}}(B_1, \mathcal{M}_{n+1}\otimes \Lambda^2 \R^4)$ such that 
$$ \nabla \Delta A + \Delta A V -\nabla A w + A (\nabla \omega +F)= \mathrm{curl} B,$$
and
$$\Vert A\Vert_{W^{2,2}} + \text{dist}\,(A, \mathcal{SO}_{n+1}) + \Vert B\Vert_{W^{1,\frac{4}{3}}} \leq C \left(\Vert V \Vert_{W^{1,2}} + \Vert w\Vert_2 +\Vert \omega \Vert_2 +    \Vert F \Vert_{L^2 \cdot W^{1,2}}\right).$$
\end{thm}

Thanks to the previous theorem, we are in position to rewrite equations in approximate form of \eqref{nf} in divergence form.
\begin{thm}[{\cite[Theorems 1.3 and 1.5]{LR}}]
\label{TLR} 
There exists $\eps>0$ and $C>0$ depending only on $\mathcal{N}$, such that if  $u \in W^{2,2}(B_1,\R^{n+1})$  satisfies
$$\Delta^2 u = \Delta (V \nabla u) + div( w \nabla u) + \nabla \omega \nabla u + F\nabla u + f ,$$ 
where $V \in W^{1,2}(B_1, \mathcal{M}_{n+1} \otimes \Lambda^{1} \R^4)$, $w\in L^2(B_1, \mathcal{M}_{n+1})$, $\omega \in L^2(B_1,\mathrm{s}o_{n+1})$, $F \in L^2\cdot W^{1,2}(B_1,  \mathcal{M}_{n+1} \otimes \Lambda^{1} \R^4)$ and  $f \in L^1(B_1,\R^{n+1})$  with
$$\Vert V \Vert_{W^{1,2}} +\Vert w \Vert_2 + \Vert \omega \Vert_2 + \Vert F \Vert_{L^2 \cdot W^{1,2}} < \eps,$$
then there exists $A\in L^\infty \cap W^{2,2}(B_1, \mathcal{G}l_{n+1})$ and $B\in  W^{1,\frac{4}{3}}(B_1, \mathcal{M}_{n+1}\otimes \Lambda^2 \R^4)$ such that 
\beq
\label{E1}
\Vert A\Vert_{W^{2,2}} + d(A, \mathcal{SO}_{n+1}) + \Vert B\Vert_{W^{1,\frac{4}{3}}} \leq C \left(\Vert V \Vert_{W^{1,2}} + \Vert w\Vert_2 +\Vert \omega \Vert_2 +    \Vert F \Vert_{L^2 \cdot W^{1,2}}\right)
\eeq
and
$$\Delta(A \Delta u)= div\left(2\nabla A \Delta u-\Delta A \nabla u + A w \nabla u + \nabla A(V\nabla u)-A \nabla (V \nabla u) -B\nabla u\right) + Af.$$
\end{thm}

A first consequence of the Theorem \ref{TLR}, is the $\eps$-regularity for approximate intrinsic and extrinsic bi-harmonic maps into $\mathbf{S}^n$. This is a refined version of Theorem 2.3 of Laurain-Rivi\`{e}re \cite{LaR}, since here we only assume the smallness on the bi-energy (rather than $\Vert \nabla^2 u\Vert^2_2 + \Vert \nabla u \Vert^4_4$) and use the fact that the map take values in the sphere $\mathbf{S}^n$.
\begin{thm}
\label{ereg}
There exist $\eps>0$, $0<\delta<1$, $\alpha>0$ and $C>0$ independent of $u$ such that if $u\in W^{2,2}(B_1, \mathbf{S}^n)$ is a solution of 
\beq
\label{nf2}
\Delta^2 u = \Delta (V \nabla u) + div( w \nabla u) + \nabla \omega \nabla u + F\nabla u +f ,
\eeq
where $V \in W^{1,2}(B_1, \mathcal{M}_{n+1} \otimes \Lambda^{1} \R^4)$, $w\in L^2(B_1, \mathcal{M}_{n+1})$, $\omega \in L^2(B_1, \mathrm{s}o_{n+1})$, $F \in L^2\cdot W^{1,2}(B_1,  \mathcal{M}_{n+1} \otimes \Lambda^{1} \R^4)$ and $f\in L^q (B_1, \R^{n+1})$ with $q>1$, which satisfy \eqref{cc} and
\begin{equation}\label{smallbien01}
\Vert \Delta u\Vert_{L^2(B_1)} \leq \eps,
\end{equation}
then we have $u\in W_{loc}^{3, 4/3} (B_1, \R^{n+1})$ and
\beq
\label{F}
\Vert \nabla^3 u \Vert_{L^{\frac{4}{3}}(B(p, \rho))} +\Vert \nabla^2 u \Vert_{L^2(B(p, \rho))} +\Vert \nabla u \Vert_{L^4(B(p, \rho))} \, \leq\, C \rho^\alpha \left(\Vert \Delta u\Vert_{L^2(B_1)}+ \Vert f\Vert_{L^q(B_1)}\right)
\eeq
for all $p\in B_{\frac{1}{2}}$ and $0\leq  \rho \leq \delta$. Moreover, if $f\in L^{4,1}(B_1)$ (Lorentz space, see e.g. \cite{Hu66}, \cite{He02}) then  $u\in W^{3,\infty} (B_{\frac{1}{16}}, \R^{n+1})$ and for $l = 1, 2, 3$ we have
\beq
\label{Fereg} \vert \nabla^l u\vert(0) \leq C_{l}  \left(\Vert \Delta u\Vert_{L^2(B_1)}+\Vert f \Vert_{L^{4,1}(B_1)} + \Vert f\Vert^2_{L^{4,1}(B_1)}\right)
\eeq
for some constant $C_l>0$. In particular, by rescaling we have for $x\in B_1$ and $l=1,2,3$:
\beq
\label{PTWSest} 
\vert \nabla^l u\vert(x) \leq \frac{C_{l}}{(1-|x|)^l}  \left(\Vert \Delta u\Vert_{L^2(B_1)}+\Vert f \Vert_{L^{4,1}(B_1)} + \Vert f\Vert^2_{L^{4,1}(B_1)}\right)\,.
\eeq
\end{thm}

\begin{proof} First of all, in order to apply Theorem \ref{TLR}, we need $\Vert \nabla^2  u \Vert_2^2 + \Vert \nabla  u \Vert_4^4$ to be small, and therefore we have to control $\Vert \nabla^2 u\Vert_{2}$ up to reducing the size of the ball from the assumption \eqref{smallbien01}. Since it is very important that all estimates are independent of the size of the ball, we give a proof on a ball of radius $r \in (0,1]$. Let $u=\xi +\eta$ where $\xi \in W^{2,2}_0(B_r)$ and $\eta \in W^{2,2}(B_r)$ be such that 
$$ \Delta \xi =\Delta u$$
and
$$ \Delta \eta=0.$$
Thanks to the standard $L^p$ theory (see e.g. \cite[Theorem 4, \S 6.3]{Evans}), we have 
\beq
\label{me}
\Vert \nabla^2 \xi \Vert_{2} \leq C\Vert \Delta u\Vert_{2} \quad \text{on } B_{r}\,.
\eeq
By classical theory of harmonic function, see e.g. \cite[Corollary 1.37]{HL}, we have
$$ \int_{B_{\frac{3r}{4}}} \vert \nabla^2 \eta \vert^2 \, dx \leq \frac{C}{r^2} \int_{B_r} \vert \nabla \eta \vert^2 \, dx\,,$$
where $C>0$ is a universal constant. Hence, by harmonicity of $\eta$ and $\xi =0$ on $\partial B_r$ (so that $\Vert \nabla u \Vert_2^2 = \Vert \nabla \xi \Vert_2^2 + \Vert \nabla \eta \Vert_2^2$ on $B_r$), we have
$$ 
\int_{B_{\frac{3r}{4}}} \vert \nabla^2 \eta \vert^2 \, dx \leq \frac{C}{r^2} \int_{B_r} \vert \nabla u\vert^2 \, dx, $$
and thus
$$ \int_{B_{\frac{3r}{4}}} \vert \nabla^2 \eta \vert^2 \, dx \leq C \left(\int_{B_r} \vert \nabla u\vert^4 \, dx\right)^\frac{1}{2}.$$
Finally, using the fact that $u$ takes values in $\mathbf{S}^n$, we automatically have that
$$ \Vert \nabla u\Vert_4^4 \leq \Vert \Delta u \Vert_2^2 ,$$
which insures that (assuming $\Vert \Delta u\Vert_2$ is small)
\beq
\label{est}
\Vert \nabla^2  u \Vert_2^2 + \Vert \nabla  u \Vert_4^4 \leq C\Vert \Delta u\Vert_2 \quad \text{on } B_r\,.
\eeq

Now assuming that $\Vert \nabla^2  u \Vert_2^2 + \Vert \nabla  u \Vert_4^4$ is small on $B_\frac{3}{4}$, thanks to \eqref{cc} and \eqref{est}, hypothesis of  Theorem \ref{TLR} are satisfied on $B_\frac{3}{4}$. Hence we can rewrite our equation as
$$ \Delta (A\Delta u)= \text{div} (K) + Af,$$
where $A\in L^\infty \cap W^{2,2}(B_\frac{3}{4}, \mathcal{G}l_{n+1})$ and $K\in L^2\cdot W^{1,2} (B_\frac{3}{4}) \subset L^{\frac{4}{3},1}(B_\frac{3}{4})$ satisfy
$$\Vert A\Vert_{W^{2,2}} + d(A, \mathcal{SO}_{n+1}) \leq C \left(\Vert \nabla^2 u\Vert_2+ \Vert \nabla u\Vert_4\right)$$
and 
$$
 \Vert K\Vert_{L^{\frac{4}{3},1}} \leq C \left(\Vert \nabla^2 u\Vert_2^2+ \Vert \nabla u\Vert_4^2\right)\,,
$$
where $C>0$ is independent of $u$.

Now let $p\in B_{\frac{1}{2}}$ and $0<\rho<\frac{1}{4}$ so that $B_\rho(p) \subset B_{\frac{3}{4}}$. Use the Hodge decomposition we decompose $A\Delta u$ on $B_\rho(p)$ as  $A\Delta u= R+ S$, where $R\in W_0^{1,2}(B_\rho(p))$ and $S\in W^{1,2}(B_\rho(p))$, such that $R$ satisfies
$$\Delta R = \text{div} (K) + Af $$
and $S$ satisfies
$$
\Delta S = 0
$$
on $B_\rho(p)$. Thanks to the standard $L^p$-theory and Sobolev embeddings, on $B_\rho(p)$ we get (using that $q>1$)
\beq
\label{EC}
\Vert R \Vert_2  \leq C \left(\Vert K\Vert_{\frac{4}{3}}+  \Vert f\Vert_\frac{q+1}{2}\right)  \leq C \left(\eps^{\frac{1}{4}}\left(\Vert \nabla^2 u \Vert_{2} + \frac{1}{\rho} \Vert \nabla u \Vert_2\right) +  \rho^\frac{q-1}{q(q+1)} \Vert f\Vert_q\right),
\eeq
where $C>0$ is independent of $u$. Now using the fact that $S$ is harmonic, thanks to Lemma \ref{hf} we have that $\gamma \mapsto \frac{1}{(\gamma\rho)^4} \int_{B_{\gamma\rho}(p)} \vert S\vert^2\, dx$  is an increasing function and hence for all $\gamma\in(0,1)$
we have
\beq
\label{EC2}
\int_{B_{\gamma \rho}(p)}  \vert S\vert^2\, dx  \leq \gamma^4 \int_{B_\rho(p)}  \vert S\vert^2\, dx. 
\eeq
We then decompose $u$ as follows :  $u= E+ F$ where $E\in W_0^{1,4}(B_\rho(p))$ and $F\in W^{1,4}(B_\rho(p))$ satisfy
$$
\Delta E = A^{-1}(R+S)\hbox{ on } B_\rho(p)$$
and
$$
\Delta F = 0 \hbox{ on } B_\rho(p) .$$
Thanks again to the standard $L^p$-theory and Sobolev embeddings, on $B_\rho(p)$ we get
\beq
\label{EC3}
\frac{1}{\rho} \Vert \nabla E \Vert_2  \leq C \left(\Vert R \Vert_2 + \Vert S \Vert_2\right)\,,
\eeq
where $C>0$ is independent of $u$. Note that $\Vert \nabla u \Vert_2 = \Vert \nabla E \Vert_2 + \Vert \nabla F \Vert_2$ on $B_\rho(p)$.

Now the function  $\gamma \mapsto \frac{1}{(\gamma\rho)^4} \int_{B_{\gamma\rho}(p)} \vert \nabla F \vert^2\, dx$  is increasing since $F$ is harmonic and  we have again, for all $\gamma\in(0,1)$,
\beq
\label{EC4}
\frac{1}{(\gamma\rho)^2}\int_{B_{\gamma\rho}(p)}  \vert \nabla F \vert^2\, dx  \leq \frac{\gamma^2}{\rho^2} \int_{B_\rho(p)}  \vert \nabla F \vert^2\, dx.
\eeq
Then, thanks to (\ref{EC}), (\ref{EC2}), (\ref{EC3}) and (\ref{EC4}), for  $\gamma$ and $\eps$ small enough (with respect to some constant independent of $u$), we have 
\begin{equation}
\int_{B_{\gamma \rho}(p)} \left(\vert \nabla^2 u \vert^2+\frac{\vert \nabla u \vert^2}{(\gamma\rho)^2}\right) \, dx \leq \frac{1}{2} \int_{B_{\rho}(p)} \left(\vert \nabla^2 u \vert^2 + \frac{\vert \nabla u \vert^2}{\rho^2}\right) \, dx + C(\gamma\rho)^\frac{2(q-1)}{q(q+1)}  \Vert f\Vert_q^2\,.
\end{equation}
Here, we have used that the $L^2$ norm of Hessian is controlled by the $L^2$ norms of the Laplacian and the gradient, up to reducing the size of the ball, see \eqref{est}. Iterating this inequality gives the following Morrey type estimate: there exists $0<\delta<\frac{1}{4}$, $\alpha>0$ and $C>0$ independent of $u$ such that 
\beq
\label{F1}
\sup_{p\in B_\frac{1}{2} , 0<\rho< \delta} \rho^{-\alpha}\left(\int_{B_{\rho}(p)} \left(\vert \nabla^2 u \vert^2+\frac{1}{\rho^2}\vert \nabla u \vert^2\right) \, dx\right) \leq C\left(\Vert \Delta u\Vert_{L^2(B_1)}^2+ \Vert f\Vert_{L^q(B_1)}^2\right)\,.
\eeq
Then, for any $p\in B_\frac{1}{2}$ and $0<\rho< \delta$ we have
$$\Vert K \Vert_{L^{\frac{4}{3}}(B(p, \rho))} \leq C\rho^\alpha \left(\Vert \Delta u\Vert_{L^2(B_1)}^2+ \Vert f\Vert_{L^q(B_1)}^2\right)\,.$$
Setting $A\Delta u= R+ S $ on  $B(p, \rho')$ with $\rho'\in \left(\frac{3\rho}{4}, \rho\right)$ as before ($S$ is harmonic), where $\rho'$ will be fixed later, we get 

\beq
\label{C1}
\Vert \nabla R\Vert_{L^{\frac{4}{3}}\left(B(p, \frac{\rho}{2})\right)} \leq C\rho^\alpha\left(\Vert \Delta u\Vert_{L^2(B_1)}^2+ \Vert f\Vert_{L^q(B_1)}^2\right).
\eeq
Using a Green formula, we get for all $y\in B(p, \frac{\rho}{2})$ that 
\begin{align}
\vert \nabla S(y) \vert & \leq \frac{C}{\rho^4} \int_{\partial B(p, \rho')}   \vert A \Delta u\vert \; dx \notag \\
& \leq \frac{C \rho^\frac{3}{2}}{\rho^4} \left( \int_{\partial B(p, \rho')}   \vert A \Delta u\vert^2 \; dx\right)^\frac{1}{2} \notag\\ 
& \leq \frac{C \rho^\frac{3}{2}}{\rho^4} \left( \frac{4}{\rho}\int_{ B(p, \rho) \setminus B(p, \frac{3\rho}{4})}   \vert A \Delta u\vert^2 \; dx\right)^\frac{1}{2},
\end{align}
the last inequality is obtained thanks to the mean value theorem, by choosing $\rho'$ correctly. Thanks to (\ref{E1}) and (\ref{F1}), for any $p\in B_\frac{1}{2}$ and all $y\in B(p, \frac{\rho}{2})$ we get  

$$\vert \nabla S(y)  \vert \leq C\rho^{\frac{\alpha}{2} -3} \left(\Vert \Delta u\Vert_{L^2(B_1)}+ \Vert f\Vert_{L^q(B_1)}\right)$$
and 
\beq
\label{D1}
 \Vert \nabla S \Vert_{L^{\frac{4}{3}}(B(p, \frac{\rho}{2}))} \leq C\rho^\frac{\alpha}{2} \left(\Vert \Delta u\Vert_{L^2(B_1)}+ \Vert f\Vert_{L^q(B_1)}\right) .
 \eeq
Thanks to (\ref{C1}) and (\ref{D1}), we get for any $p\in B_{\frac{1}{2}}$ and all $0<\rho< \frac{\delta}{4}$,
\beq
\label{F2}
\left( \int_{B(p, \frac{\rho}{2}))}  \vert \nabla(A\Delta u)\vert^\frac{4}{3} \; dx\right)^{\frac{3}{4}} \leq  C \rho^\frac{\alpha}{2}\left(\Vert \Delta u\Vert_{L^2(B_1)}+ \Vert f\Vert_{L^q(B_1)}\right).
\eeq
Finally thanks to (\ref{F1}) and (\ref{F2}) we get (\ref{F}).\\

Then we can bootstrap those estimates so that there exists $\beta>0$ such that 
\beq
\sup_{p\in B_\frac{1}{4} , 0<\rho <\frac{\delta}{8}} \rho^{-\beta} \int_{B_{\rho}(p)} \vert \Delta^2 u \vert \, dx   \leq C \left(\Vert \Delta u\Vert_{L^2(B_1)}^2+ \Vert f\Vert_{L^q(B_1)}^2\right).
\eeq
Now suppose $f\in L^{4,1}(B_1)$ instead of $L^q(B_1)$ for some $q>1$ and the above estimates are still valid by replacing $\Vert f \Vert_{L^q(B_1)}$ with $\Vert f \Vert_{L^{4,1}(B_1)}$ . Then a classical Green formula gives, for all $p\in B_{\frac{1}{8}}$
\be
\begin{split}
&|\nabla^3 u|(p)\le C \frac{1}{|x-p|^3}\ast \chi_{B_{\frac{{1}}{4}}}\ |\Delta^2 u|+C \left(\Vert \Delta u\Vert_{L^2(B_1)}+ \Vert f\Vert_{L^{4,1}(B_1)}\right),\\
&|\nabla^2 u|(p)\le C \frac{1}{|x-p|^2}\ast \chi_{B_{\frac{1}{4}}}\ |\Delta^2 u|+C  \left(\Vert \Delta u\Vert_{L^2(B_1)}+ \Vert f\Vert_{L^{4,1}(B_1)}\right),\\
& |\nabla u|(p)\le C \frac{1}{|x-p|}\ast \chi_{B_{\frac{1}{4}}}\ |\Delta^2 u|+C \left(\Vert \Delta u\Vert_{L^2(B_1)}+ \Vert f\Vert_{L^{4,1}(B_1)}\right)\,,
\end{split}
\ee
where $\chi_{B_{\frac{1}{4}}}$ is the characteristic function of the ball $B_{\frac{1}{4}}$. We use the Green formula on a cut-off of u, the remaining terms are easily control thanks to (\ref{F}). Together with injections proved by Adams in \cite{Ad}, see also exercise 6.1.6 of \cite{Gra2}, the latter shows that
$$
\Vert \nabla^3 u \Vert_{L^{r}\left(B_\frac{1}{8}\right)}+ \Vert \nabla^2 u \Vert_{L^{r}\left(B_\frac{1}{8}\right)} + \Vert \nabla u \Vert_{L^{r}\left(B_\frac{1}{8}\right)} \leq C  \left(\Vert \Delta u\Vert_{L^2(B_1)}+ \Vert f\Vert_{L^{4,1}(B_1)} + \Vert f\Vert^2_{L^{4,1}(B_1)}\right),
$$
for some $r>4/3$. Then bootstrapping this estimate, we get 
$$\Vert  \nabla^3 u\Vert_{L^{\bar{q}} \left(B_\frac{1}{16}\right)} + \Vert  \nabla^2 u\Vert_{L^{\bar{p}} \left(B_\frac{1}{16}\right)}  + \Vert  \nabla u\Vert_{L^{\bar{p}}\left(B_\frac{1}{16}\right)}  \leq C \left(\Vert \Delta u\Vert_{L^2(B_1)}+ \Vert f\Vert_{L^{4,1}(B_1)} + \Vert f\Vert^2_{L^{4,1}(B_1)}\right),$$
where $\bar{q}$ is the limiting exponent of the bootstrapping given by the Sobolev injection of $W^{1,q}$ into $L^{\bar{q}}$. Indeed, thanks to (\ref{cc}), the only limiting term for the bootstrap is the regularity of $f$. But thanks to the embedding of $W^{1,(4,1)}$ into $L^\infty$, see e.g. \cite{KKM99} and \cite{Tar98}, we can conclude the proof of the theorem.
\end{proof}

\begin{lemma}
\label{hf}
 Let $v$ be a harmonic function on $B_1$. For every point $p$ in $B_1$, the function
$$\rho \mapsto \frac{1}{\rho^4}  \int_{B(p,\rho)} \vert  v\vert^2 \; dx$$
is increasing.
\end{lemma}
\begin{proof}
\beq
\label{hf1}
\frac{d}{d\rho} \left[  \frac{1}{\rho^4}  \int_{B(p,\rho)} \vert  v\vert^2 \; dx \right]= \frac{-4}{\rho^5}  \int_{B(p,\rho)} \vert v\vert^2 \; dx +\frac{1}{\rho^4}  \int_{\partial B(p,\rho)} \vert  v\vert^2 \; d\sigma .
\eeq
Let $(\phi_k^l)_{l,k}$ be an $L^2$-basis of eigenfunctions of the Laplacian on $\mathbf{S}^3$. In particular
$$\Delta \phi^l_k=-l(l+2)\phi_k^l$$
We have
$$v(\rho, \theta) =\sum_{l=0}^{+\infty} \sum_{k=1}^{N_l} a_k^l   \phi_k^l,  \hbox{ on } \partial B(p,\rho)$$
where $N_l$ is the dimension of the eigenspace corresponding to $-l(l+2)$, see \cite{Groe}. Hence
\beq
v(r,\theta) =\sum_{l=0}^{+\infty} \sum_{k=1}^{N_l}  a_k^l \left(\frac{r}{\rho}\right)^l  \phi_k^l, \hbox{ on } B(p,\rho). 
\eeq
Then
\beq
\label{hf2}
\int_{\partial B(p,\rho)} \vert v\vert^2 \; dx =\sum_{l=0}^{+\infty} \sum_{k=1}^{N_l}  \vert a_k^l \vert^2 \rho^3
\eeq
and 
\beq
\label{hf3}
 \int_{B(p,\rho)} \vert  v\vert^2 \; dx=\sum_{l=0}^{+\infty} \sum_{k=1}^{N_l}  \frac{ \vert a_k^l \vert^2}{2l+4} \rho^4.
 \eeq
Finally putting (\ref{hf1}), (\ref{hf2}) and (\ref{hf3}) together we get the desired result.
\end{proof}


\section{Further remarks and open questions}\label{APPB}

By the proof of Theorem \ref{ereg} we know that $\|\Delta^2 u\|_{L^1_{loc}(B_1)}$ is small for a $W^{2,2}$ weakly intrinsic or extrinsic bi-harmonic map $u$ defined on $B_1$ with small bi-energy. Analogously, in Theorem A.4 of Lamm and the second author's work \cite{LL13} for weakly harmonic maps, $|\Delta u| \simeq |\nabla u|^2$ is estimated to be in the local Hardy space $h^1 \subsetneq L^1$ for weakly harmonic maps with small Dirichlet energy on the $2$-disk. In this case, the improved global Hardy estimate turns out to be equivalent to the use of the (1st order) Hardy inequality plus the $\eps$-regularity $|\nabla u| \leq C\varepsilon_0/(1-|x|)$ for weakly harmonic maps on the $2$-disk. It is worth remarking that such improved global estimate is a typical compensation phenomenon for the special Jacobian structure of the harmonic map equation
$$
-\Delta u \,=\, A(u)(\nabla u, \nabla u) \,=\, \Omega \cdot \nabla u\,,
$$
where $\Omega$ is anti-symmetric. Given the special structure of the intrinsic or extrinsic bi-harmonic map equation, it is quite natural to conjecture that $|\Delta^2 u|$ is in $L^1(B_1)$ (with small $L^1$-norm), but as far as we know this has not yet been proved. We leave it as an interesting open question. If this was true, given the $\eps$-regularity Theorem \ref{ereg} and the following theorem (Theorem \ref{appblemma}), then one can bypass the use of the Hardy inequality (Theorem \ref{hardyineq}) in the proof of the energy convexity results (Theorems \ref{THM1}, \ref{extrinsicbiharmonic} and \ref{THM2}), which we will explain below. In a forthcoming paper, we will explore more on this interesting relations between the Hardy inequality, $\eps$-regularity, compensation phenomenon of the structure of the bi-harmonic map equations and the global integrability of the solution $u$.

\begin{thm} \label{appblemma} Let $f\geq 0$ satisfy $$f(x) \leq \frac{C_0}{(1-|x|)^4} \quad \text{a.e.}\,\, x\in B_1 \quad \text{and}\quad \|f\|_1 \leq C_0$$ for some constant $C_0>0$. Then there exists a function $\psi \in L^\infty\cap W^{2,2}(B_1)$ solving the boundary value problem
 \begin{equation}
  \left\{
   \begin{aligned}
     \Delta^2 \psi\,&=\, f && \text{in }\, B_1\,, \\
    \psi\,&=\, \partial_\nu \psi \,= \,0 && \text{on }\, \partial B_1\,.\\
   \end{aligned}
 \right.
   \end{equation}
Moreover, there exists a constant $C>0$ such that
\begin{equation}\label{psii1}
\Vert \Delta \psi \Vert_2+ \Vert \nabla \psi\Vert_4 +\Vert \psi\Vert_\infty \,\leq\, C \,C_0\,.
\end{equation}
\end{thm}

\begin{proof}
The idea of the proof follows \cite[Proposition 1.68]{Sem}.  Since the Green's function of $\Delta^2$ on $B_1$ with the clamped plate boundary values is given explicitly by
$$
G(x,y) = c \left(\ln|x-y| - \ln\left(\left|\frac{x}{|x|} - |x|y\right|\right) - \frac{|x-y|^2}{2 \left|\frac{x}{|x|} - |x|y\right|^2}  + \frac{1}{2}\right)
$$
for some normalizing constant $c<0$ (see \cite{B1905} or \cite[Lemma 2.1]{GS97}), we can write
\begin{align}\label{psicon}
\psi(x) &= c\int_{B_1} f(y) \left(\ln|x-y| - \ln\left(\left|\frac{x}{|x|} - |x|y\right|\right) - \frac{|x-y|^2}{2 \left|\frac{x}{|x|} - |x|y\right|^2}  + \frac{1}{2}\right)dy\,.
\end{align}

Let $\theta\in C^{\infty}_0(B_1)$ be a smooth bump function such that $0\leq \theta \leq 1$, $\theta=1$ in $B_{\frac{1}{16}}$ and $\text{spt}(\theta)\subset B_{\frac{1}{8}}$. For $x\in B_1$ we define
\begin{equation}\label{lfunc}
l_x(y) := \sum_{j=0}^{\infty}\theta\left(2^j(1-|x|)^{-1}(x-y)\right) \quad \text{for }y\in B_1\,.
\end{equation}
We claim that for any $x,y\in B_1$
\begin{equation}\label{Lest2}
-20\ln2 \leq \ln|x-y| - \ln\left(\left|\frac{x}{|x|} - |x|y\right|\right) - \frac{|x-y|^2}{2 \left|\frac{x}{|x|} - |x|y\right|^2}  + \frac{1}{2}+ l_x(y)\ln2  \leq 20\ln 2\,,
 \end{equation}
To see this, it is clear that for $x,y\in B_1$ such that
\begin{equation}\label{Temp-9090-2}
2^{-k}\leq |x-y|\leq 2^{-k+1}, \quad k \in \N_0
\end{equation}
we have
\begin{equation}\label{theta4409909}
-k\ln 2\,\leq \ln|x-y|\,\leq (-k+1)\ln 2\,.
\end{equation}
Now note that
$$
1-|x|-|x-y| \leq 1-|x|+|x|-|y|= 1-|y| \leq 1-|x|+|x-y|\,,
$$
and therefore for $x\in B_{1-2^{-i-1}} \setminus  B_{1-2^{-i}}$, i.e., $1-|x| \in [2^{-i-1},2^{-i}], i \in \N_0$ (with $\bar{B}_0 = \emptyset$) and any $y\in B_1$ satisfying \eqref{Temp-9090-2}, we have
$$
1-|y| \,\in\,
\left\{
   \begin{aligned}
     &\left[2^{-i-1}-2^{-k+1},2^{-i}+2^{-k+1}\right]\quad &&\text{if } k\geq i+4;\\
    &\left[0,2^{-i}+2^{-k+1}\right]\quad &&\text{if } k\leq i+3.\\
   \end{aligned}
 \right.
$$
We also have
\begin{align}\label{TEMO2000009}
0 \leq (1-|x|)(1-|y|) &\leq (1-|x|^2)(1-|y|^2) \notag\\
   &= \left|\frac{x}{|x|}- |x|y\right|^2 - |x-y|^2 \leq 2^2(1-|x|)(1-|y|)\,,
\end{align}
and thus
$$
\left|\frac{x}{|x|}- |x|y\right|^2 - |x-y|^2 \in
\left\{
   \begin{aligned}
     &\left[2^{-2i-2}-2^{-i-k},2^{-2i+2}+2^{-i-k+3}\right]\quad &&\text{if } k\geq i+4;\\
    &\left[0,2^{-2i+2}+2^{-i-k+3}\right]\quad &&\text{if } k\leq i+3.\\
   \end{aligned}
 \right.
$$
Combining this with \eqref{Temp-9090-2} we get
$$
\left|\frac{x}{|x|}- |x|y\right|^2 \in\left\{
   \begin{aligned}
     &\left[2^{-2i-2}-2^{-i-k} + 2^{-2k},2^{-2i+2}+2^{-i-k+3}+2^{-2k+2} \right] \text{if } k\geq i+4;\\
    &\left[2^{-2k},2^{-2i+2}+2^{-i-k+3}+2^{-2k+2}\right]\,\ \text{if } k\leq i+3.\\
   \end{aligned}
 \right.
$$

Now using the facts that for $k\geq i+4$ we have
$$
2^{-2i-2}-2^{-i-k} + 2^{-2k} \geq 2^{-2i-4} \quad \text{and}\quad 2^{-2i+2}+2^{-i-k+3}+2^{-2k+2}  \leq 2^{-2i+4}
$$
and for $k \leq i+3$ we have
$$
2^{-2i+2}+2^{-i-k+3}+2^{-2k+2}  \leq 2^{-2k+10}\,,
$$
we arrive at
$$
\left|\frac{x}{|x|}- |x|y\right|^2 \,\in\, \left\{
   \begin{aligned}
     &\left[2^{-2i-4},2^{-2i+4}\right]\quad &&\text{if } k\geq i+4;\\
    &\left[2^{-2k},2^{-2k+10}\right]\quad &&\text{if } k\leq i+3,\\
   \end{aligned}
 \right.
$$
and hence
\begin{equation}\label{Lfunction1}
-\ln \left|\frac{x}{|x|}- |x|y\right|\, \in \,\left\{
   \begin{aligned}
     &[(i-2)\ln2, (i+2)\ln2]\quad &&\text{if } k\geq i+4;\\
    &[(k-5)\ln2, k\ln2]\quad &&\text{if } k\leq i+3.\\
   \end{aligned}
 \right.
\end{equation}
Combining \eqref{theta4409909} and \eqref{Lfunction1} we get
\begin{equation}\label{Lfunction1-1-1}
\ln|x-y| - \ln\left(\left|\frac{x}{|x|} - |x|y\right|\right) \in \left\{
   \begin{aligned}
     &[(-k+i-2)\ln2, (-k+i+3)\ln2]\quad &&\text{if } k\geq i+4;\\
    &[-5\ln2, \ln 2] \quad(\text{in fact, }\,[-5\ln2, 0])\quad &&\text{if } k\leq i+3,\\
   \end{aligned}
 \right.
\end{equation}
and
\begin{equation}\label{Lfunction1-1-100}
\frac{|x-y|^2}{2 \left|\frac{x}{|x|} - |x|y\right|^2} \in \left\{
   \begin{aligned}
     &\left[0, 1/4\right]\quad &&\text{if } k\geq i+4;\\
    &\left[2^{-6}, 1\right] \quad &&\text{if } k\leq i+3,\\
   \end{aligned}
 \right.
\end{equation}
for any $x\in B_{1-2^{-i-1}} \setminus  B_{1-2^{-i}}, i\geq 0$, and any $y\in B_1$ satisfying \eqref{Temp-9090-2} for some $k\geq 0$.

Now for any $x\in B_{1-2^{-i-1}} \setminus  B_{1-2^{-i}}, i\geq 0$, and any $y\in B_1$ satisfying \eqref{Temp-9090-2}, since $0\leq \theta \leq 1$, $\theta=1$ in $B_{\frac{1}{16}}$ and $\text{spt}(\theta)\subset B_{\frac{1}{8}}$, we get that for any $j\geq 0$
$$
\theta\left(2^j(1-|x|)^{-1}(x-y)\right)= 0  \quad \text{for}\quad  |x-y|\ge  2^{-j-3}(1-|x|) \in [2^{-j-i-4}, 2^{-j-i-3}]
$$
and
$$
\theta\left(2^j(1-|x|)^{-1}(x-y)\right) = 1 \quad \text{for}\quad  |x-y|\leq 2^{-j-4}(1-|x|)\in [2^{-j-i-5}, 2^{-j-i-4}]\,.
$$
Therefore (combining with \eqref{Temp-9090-2}),
\begin{equation}\label{theta229090}
\theta\left(2^j(1-|x|)^{-1}(x-y)\right)= 0  \quad \text{for}\quad j \geq k-i-3
\end{equation}
and
\begin{equation}\label{theta110909}
\theta\left(2^j(1-|x|)^{-1}(x-y)\right)= 1  \quad \text{if}\quad k-1 \geq j+ i+5 \quad(\text{i.e. }\, j\leq k-i-6)\,.
\end{equation}

Hence for any $x\in B_{1-2^{-i-1}} \setminus  B_{1-2^{-i}}, i\geq 0$ and any $y\in B_1$ such that $2^{-k}\leq |x-y|\leq 2^{-k+1}$ for some $k=0,1,2,...,$ \eqref{lfunc}, \eqref{theta229090} and \eqref{theta110909} imply
\begin{equation}\label{theta559999999}
\left\{
   \begin{aligned}
     &k-i - 10\,\leq\,l_x(y)\,\leq\, k-i+10 \quad &&\text{if } k\geq i+4;\\
    & l_x(y)=0 \quad &&\text{if } k\leq i+3\,.\\
   \end{aligned}
 \right.
\end{equation}
Combining \eqref{Lfunction1-1-1}, \eqref{Lfunction1-1-100} and \eqref{theta559999999} gives \eqref{Lest2}.

Therefore, in order to obtain the $L^\infty$-bound of $\psi$ on $B_1$ as in \eqref{psii1}, it suffices to bound $\int_{B_1} f(y) l_x(y) dy$ since we have \eqref{psicon} and \eqref{Lest2}. Therefore, using the facts that $f\geq 0, \,0\leq \theta \leq 1$ and $\text{spt}(\theta)\subset B_{\frac{1}{8}}$, for any $x\in B_1$ we have
\begin{align}\label{FINALLLY}
&\left| \int_{B_1} f(y) l_x(y)dy\right| \leq \sum_{j=0}^{\infty} \int_{B_1} f(y) \theta(2^j(1-|x|)^{-1}(x-y))dy\notag\\
=\,& \sum_{j=0}^{\infty}\int_{B_{2^{-j-3}(1-|x|)}(x)} f(y)\theta\left(2^j(1-|x|)^{-1}(x-y)\right)dy\notag\\
\leq \,& \sum_{j=0}^{\infty} \int_{B_{2^{-j-3}(1-|x|)}(x)} \frac{C_0}{(1-|y|)^4}dy \notag\\
\leq \,& \sum_{j=0}^{\infty} \int_{B_{2^{-j-3}(1-|x|)}(x)} \frac{C_0}{\left(\frac{7}{8}\right)^4(1-|x|)^4}dy \notag\\
 \leq \,& C_0\left(\frac{8}{7}\right)^4\frac{\pi^2}{2}\sum_{j=0}^{\infty} 2^{-4j-12} \,\leq\, CC_0\,.
\end{align}



Combining \eqref{psicon}, \eqref{Lest2}, \eqref{FINALLLY} and the fact that  $\|f\|_1\leq C_0$ yields
$$
|\psi(x)| \leq  CC_0\,.
$$
This gives the desired $L^{\infty}$-bound of $\psi$ on $B_1$. The $L^2$-estimate for $\Delta \psi$ and $L^4$-estimate for $\nabla \psi$ simply follows from an integration by parts argument.
\end{proof}

Now let us explain how we can bypass the Hardy inequality (Theorem \ref{hardyineq}) in the proof of the energy convexity (Theorems \ref{THM1}, \ref{extrinsicbiharmonic} and \ref{THM2}) using Theorem \ref{appblemma} and the $\eps$-regularity Theorem \ref{ereg}, assuming that the $L^1$-norm of $\Delta^2 u$ is small for the intrinsic or extrinsic bi-harmonic maps. This is a quite natural assumption for most of the applications, for example, this is valid if $u \in W^{3,\frac{4}{3}}(B_1)$ and $\|\nabla^2 u\|_{L^2}$ is small or if the bi-harmonic map is defined on a larger set containing $B_1$.

\begin{lemma}
\label{lf}
There exist constants $\varepsilon_0>0, C>0$ such that if $u\in W^{2,2}(B_1, \mathbf{S}^n)$ is a weakly intrinsic or extrinsic bi-harmonic map with 
$$\int_{B_1}|\Delta u|^2 \, dx \leq\varepsilon_0 \quad \text{and} \quad \int_{B_1}|\Delta^2 u| \, dx \leq \varepsilon_0\,,$$ 
then there exists a solution $\phi$ of
\beq
\label{Eq1}
\left\{
\begin{aligned}
\Delta^2 \phi &= q(u) \quad \quad\; \text{ on }  B_1\,,\\
\phi & =\partial_\nu \phi =0 \quad \text{ on }  \partial B_1\,.
\end{aligned}
\right.
\eeq
such that 
\beq
\label{phipro}
 \Vert \Delta \phi \Vert_2+ \Vert \nabla \phi\Vert_4 +\Vert \phi\Vert_\infty \leq C\varepsilon_0\,.
 \eeq
Here $q(u) = |\nabla u|^4, |\nabla u|^2|\Delta u|, |\nabla^2 u|^2$ or $|\nabla u||\nabla \Delta u|$. 
\end{lemma}
\begin{proof}
Note that by the $\eps$-regularity Theorem \ref{ereg} (with $f\equiv 0$), we have
\begin{equation}
\label{Qu}
|q(u)|(x)\leq C\Vert \Delta u\Vert_2^2(1-|x|)^{-4} \leq C\varepsilon_0 (1-|x|)^{-4}\,,
\end{equation}
a.e. $x\in B_1$. We can then apply Theorem \ref{appblemma} (with $C_0 = C\varepsilon_0$) to conclude the proof of this lemma.
\end{proof}

\begin{lemma}\label{wentelemma}
There exist constants $\varepsilon_0>0, C>0$ such that if $u, v\in W^{2,2}(B_1, \mathbf{S}^n)$ with $u|_{\partial B_1}=v|_{\partial B_1},\partial_\nu u |_{\partial B_1}=\partial_\nu v|_{\partial B_1},$ 
$$\int_{B_1} \vert \Delta u\vert^2 \, dx \leq\varepsilon_0 \quad \text{and} \quad  \int_{B_1}|\Delta^2 u| \, dx \leq \varepsilon_0 \,,$$ 
and $u$ is a weakly intrinsic or extrinsic bi-harmonic map, then there exists $C>0$ such that
\begin{equation}\label{keyestimate}
\int_{B_1} \vert v-u\vert^2  q(u) \; dx  \,\leq \,C\varepsilon_0 \int_{B_1} \vert \Delta(v-u)\vert^2\;dx\,.
\end{equation}
Here $q(u) = |\nabla u|^4, |\nabla u|^2|\Delta u|, |\nabla ^2 u|^2$ or $|\nabla u||\nabla \Delta u|$.\end{lemma}

\begin{rem}This estimate replaces the direct use of the Hardy inequality (Theorem \ref{hardyineq}) in the proofs of Theorem \ref{THM1}, Theorem \ref{extrinsicbiharmonic} and Theorem \ref{THM2} and then yields new proofs assuming $\|\Delta^2 u\|_{L^1(B_1)}$ is small, but more importantly this shows that the compensation phenomena obtained in Lemma \ref{lf} using the $\epsilon$-regularity (Theorem \ref{ereg}) is almost equivalent to the Hardy inequality.
\end{rem}

\begin{proof}
Let $\phi$ be a solution of \eqref{Eq1}, then
\begin{align}
\label{I}
\int_{B_1} \vert v-u \vert^2 q(u)\; dx \,&= \,\int_{B_1} \vert v-u \vert^2 \Delta^2 \phi \; dx\; = \int_{B_1} \Delta (\vert v-u \vert^2)\Delta \phi \; dx\notag \\
& = 2\int_{B_1} \langle \Delta (v-u), v-u\rangle \Delta \phi \; dx +2  \int_{B_1} \vert \nabla (v-u)\vert^2 \Delta \phi \; dx \notag\\
& \leq 2\left( \left( \int_{B_1} \vert  \Delta (v-u)\vert^2 \; dx\right)^\frac{1}{2} \left( \int_{B_1} \vert  (v-u)\vert^2 \vert \Delta \phi\vert^2  \; dx\right)^\frac{1}{2}  + \Vert\nabla (v-u) \Vert_4^2 \Vert \Delta \phi\Vert_2 \right).
\end{align}
Remarking that we have \eqref{control1}, then Lemma \ref{lf} permits to control the second term in \eqref{I} as desired. Now for the first term in \eqref{I}, we have
\begin{align}
\label{e19}
\int_{B_1} \vert v-u\vert^2 \vert \Delta \phi \vert^2 \; dx 
=& -2 \int_{B_1} \langle \nabla  (v-u), (v-u) \nabla \phi\rangle  \Delta \phi \; dx - \int_{B_1} \vert  v-u\vert^2 \langle \nabla \phi, \nabla \Delta \phi \rangle \; dx\notag \\
=&-2 \int_{B_1} \langle \nabla  (v-u), (v-u) \nabla \phi\rangle  \Delta \phi \; dx +2 \int_{B_1} \langle \nabla  (v-u), (v-u)\phi \nabla \Delta \phi \rangle  \; dx \notag\\  
&+ \int_{B_1} \vert  v-u\vert^2  \phi  \Delta^2 \phi \; dx \notag\\
=& -4 \int_{B_1} \langle \nabla  (v-u), (v-u) \nabla \phi\rangle  \Delta \phi \; dx  -2 \int_{B_1} \vert \nabla  (v-u)\vert^2  \phi \Delta \phi \; dx \notag\\  
&-2 \int_{B_1} \langle \Delta  (v-u), v-u\rangle  \phi \Delta \phi \; dx + \int_{B_1} \vert  v-u\vert^2  \phi  \Delta^2 \phi  \; dx \notag\\
\leq &  \, C \left( \Vert \nabla \phi \Vert_4 \Vert \nabla (v-u)\Vert_4 \left( \int_{B_1} \vert v-u\vert^2 \vert \Delta \phi \vert^2 \; dx\right)^\frac{1}{2}  + \Vert \phi\Vert_\infty \Vert \Delta \phi\Vert_2 \left( \int_{B_1} \vert  \nabla  (v-u)\vert^4 \; dx\right)^\frac{1}{2} \right.\\
&+\left.  \Vert \phi\Vert_\infty\Vert \Delta  (v-u)\Vert_2 \left( \int_{B_1} \vert v-u\vert^2 \vert \Delta \phi \vert^2 \; dx\right)^\frac{1}{2} +\Vert \phi \Vert_\infty \int_{B_1} \vert  v-u\vert^2  \Delta^2 \phi  \; dx  \right)\,, \notag
\end{align}
where we used the fact $\Delta^2 \phi \geq0$ for the last inequality. Then, thanks to 
\begin{align}
\label{e20}
\int_{B_1} \vert  v-u\vert^2  \Delta^2 \phi  \; dx &= -2 \int_{B_1} \langle \nabla( v-u),(v-u)  \nabla \Delta \phi  \rangle\; dx\notag\\
& =2 \int_{B_1} \langle \Delta ( v-u),v-u\rangle   \Delta \phi  \; dx +2  \int_{B_1} \vert  \nabla (v-u)\vert^2  \Delta \phi  \; dx \notag\\
&\leq C\left(\Vert \Delta  (v-u)\Vert_2  \left( \int_{B_1} \vert v-u\vert^2 \vert \Delta \phi \vert^2 \; dx\right)^\frac{1}{2}  + \Vert \nabla (v-u)\Vert_4^2  \Vert \Delta  \phi\Vert_2\right)\,,
\end{align}
combining \eqref{phipro}, \eqref{e19} and \eqref{e20} gives
\begin{align}
\int_{B_1} \vert v-u\vert^2 \vert \Delta \phi \vert^2 \; dx  \leq C\left(\Vert \Delta u \Vert_2^2 \Vert \Delta  (v-u)\Vert_2  \left( \int_{B_1} \vert v-u\vert^2 \vert \Delta \phi \vert^2 \; dx\right)^\frac{1}{2}  + \Vert \Delta u \Vert_2^4 \Vert \Delta  (v-u)\Vert_2^2 \right)\,,
\end{align}
which finally implies that the first term in \eqref{I} is also controlled as desired if $\varepsilon_0$ is small.
\end{proof}

\addcontentsline{toc}{chapter}{Bibliographie}
\bibliographystyle{plain}
\bibliography{BH}
\end{document}